\documentclass[12pt]{amsart}
\usepackage{hyphenat}
\usepackage{amsmath}
\usepackage{amssymb}
\usepackage{amsthm}
\newtheorem{problem}{Problem}
\usepackage{amsfonts}
\usepackage{mathtools}
\usepackage{tikz-cd} 
\usepackage{array}
\usepackage{silence}
\usepackage{tikz}
\usepackage{booktabs}
\usepackage{faktor}

\usepackage{microtype} 
\usepackage{hyperref}  

\usepackage{geometry}
\newtheorem{theorem}{Theorem}[section]
\newtheorem{lemma}[theorem]{Lemma}
\newtheorem{proposition}[theorem]{Proposition}
\newtheorem{corollary}[theorem]{Corollary}

\theoremstyle{definition}
\newtheorem{openproblem}{Open Problem}

\theoremstyle{definition}
\newtheorem{definition}[theorem]{Definition}
\newtheorem{example}[theorem]{Example}
\newtheorem{remark}[theorem]{Remark}
\newcommand{\Pn}{P_n}
\newcommand{\StG}[1]{\text{St}(#1)}
\DeclareMathOperator{\Inj}{Inj}

\begin{document}


\title[The Margolis-Rhodes Monoid on a Graph]{The Margolis-Rhodes Monoid of Continuous Functions on a Graph}

\author{Stuart Margolis}
\address{Department of Mathematics \\ Bar-Ilan University \\ Ramat-Gan, Israel}
\email{margolis@math.biu.ac.il}

\author{John Rhodes}
\address{Department of Mathematics \\ University of California-Berkeley \\ Berkeley, California, USA}
\email{blvdbastille@gmail.com}


\date{}

\begin{abstract}
We investigate the structural, combinatorial, ideal-theoretic, and Krohn-Rhodes complexity of the Margolis-Rhodes monoid $\operatorname{MR}(\Gamma)$ of partial continuous functions on a finite simple graph $\Gamma$, viewed topologically as a 1-dimensional simplicial complex. Alongside the full monoid, we examine some subsemigroups including $\StG{\Gamma}$, defined by the condition that the full inverse image of any vertex is strictly an edge or the empty set, and $\Inj(\Gamma)$, the monoid of all partial $1$-$1$ continuous functions. We provide explicit combinatorial enumerations and structure for path graphs ($\Pn$) and cycle graphs ($C_n$). We compute Green's relations, showing in particular that the partial order of regular $\mathcal{J}$-classes is isomorphic to the poset of induced subgraphs of $\Gamma$. Finally, we apply these structural invariants to Krohn-Rhodes complexity theory. It is known that the Margolis-Rhodes monoid has complexity at most $2$ for $\operatorname{MR}(\Gamma)$ and $1$ for $\StG\Gamma$ and $\operatorname{Inj}(\Gamma)$. We show that for cycles $C_n$, the complexity of its Margolis-Rhodes monoid is 2 if and only if $n$ is at least 4. For path graphs, we prove that the complexity of the Margolis-Rhodes monoid of $P_n$ is 2 if $n$ is at least 14.
\end{abstract}

\maketitle

\section{Introduction}

One of the most important results of finite semigroup theory is the Krohn-Rhodes Decomposition Theorem \cite{qtheory}. This theorem proves that any finite transformation semigroup can be decomposed via wreath products into finite groups and finite aperiodic semigroups. This naturally leads to the definition of Krohn-Rhodes complexity, which measures the minimal number of groups needed in such a wreath product decomposition.

 The authors of this paper explore transformation semigroups whose fiber (inverse image of a point) size is strictly bounded. Let $(Q,S)$ be a finite transformation semigroup. The system is defined to have degree at most 2 if every transformation $s \in S$ is at most 2-to-1; that is, for every point $x \in Q$, the fiber satisfies $|x s^{-1}| \le 2$. Margolis and Rhodes \cite{cremona, deg2part2} encoded fibers geometrically: elements of the underlying set $Q$ are treated as vertices, while the pairs of points collapsed by the action define the edges of a graph called the {\em fiber graph} of the transformation semigroup. This led to the Margolis-Rhodes Representation Theorem, which proves that a finite transformation semigroup has degree at most 2 if and only if it can be faithfully embedded as a subsemigroup of a Margolis-Rhodes monoid $\operatorname{MR}(\Gamma)$ on some finite simple graph $\Gamma = (V, E)$.

Within this framework, the graph $\Gamma$ is viewed as a 1-dimensional simplicial complex whose simplices comprise single vertices $\{v\}$ and single edges $\{u, v\} \in E$. A partial transformation $f \colon V \rightarrow V$ belongs to the Margolis-Rhodes monoid $\operatorname{MR}(\Gamma)$ if and only if it satisfies a continuity condition: the full preimage of any simplex under the action must itself be a single simplex or completely empty. This condition mandates that vertex preimages  and edge preimages cannot encompass more than two vertices, connecting the size of collapse of functions to the dimension of the complex.. See \cite{cremona, deg2part2} for background on transformation semigroups of degree at most 2 and their representation as monoids of continuous functions on graphs.

Refining this further introduces subsemigroups that isolate distinct topological behaviors. We focus primarily on the subsemigroups $\StG{\Gamma}$ and $\Inj(\Gamma)$. Rather than allowing arbitrary partial continuous maps, $\StG{\Gamma}$ restricts its elements to those continuous maps where the full inverse image of any vertex is strictly a single edge or the empty set. Equivalently, $\StG{\Gamma}$ can be characterized as the collection of all maps that send the edges of a matching of size $k$ to an independent set of the same size. Given the importance of matchings and independent sets in graph theory, we are surprised that this semigroup has never been considered as far as we know.

$\Inj(\Gamma)$ is the submonoid of $\operatorname{MR}(\Gamma)$ consisting of all partial $1$-$1$ continuous maps on $\Gamma$. Equivalently, it is the monoid of all partial functions $f$ such that $f^{-1}$ induces a graph embedding from the induced subgraph on the image of $f$ to the induced subgraph on the domain of $f$.

The aim of this paper is to study $\operatorname{MR}(\Gamma)$ and $\StG{\Gamma}$ across classical graph families, providing explicit combinatorial computations and analyzing their ideal structures. We tabulate the precise orders of these monoids for path graphs ($\Pn$) and cycle graphs ($C_n$). We compute Green's relations and show that the partial order of regular $\mathcal{J}$-classes of $\operatorname{MR}(\Gamma)$ is isomorphic to the poset of isomorphism classes of induced subgraphs of $\Gamma$. On the other hand, the partial order of regular $\mathcal{J}$-classes of $\StG{\Gamma}$ is always a chain. We determine when these semigroups are regular.

We then study Krohn-Rhodes complexity. It is known that the Krohn-Rhodes complexity $\operatorname{MR}(\Gamma)c$ is at most $2$ for any graph. We prove that for any graph $\Gamma$, the Krohn-Rhodes complexity of both $\StG{\Gamma}$ and $\Inj(\Gamma)$ are at most 1. We show that $MR(C_n)c=2$ if and only if $n>3$. For paths $\Pn$, we prove that $MR(\Pn)c=2$ if $n>13$.

We then give a semigroup theoretic and an order theoretic characterization of the Margolis-Rhodes monoid. We prove that $\operatorname{MR}(\Gamma)$ is the translational hull of the aperiodic 0-simple semigroup whose structure matrix is the incidence matrix of the simplicial complex associated to $\Gamma$. This allows us to prove that $\Gamma$ is isomorphic to $\Gamma'$ is and only if $\operatorname{MR}(\Gamma)$ is isomorphic to $\operatorname{MR}(\Gamma')$. Analogous results hold for $\StG{\Gamma}$ if $\Gamma$ has no isolated vertices.

We then show that $\operatorname{MR}(\Gamma)$ is a monoid of Galois correspondences on the face poset of $\Gamma$ extended by adding a new top element. Thus continuous functions in our sense are also continuous functions in the sense of topology in the Alexandroff topology on the extended face poset.

We close with open questions including questions about the quasivariety and pseudovariety of finite semigroups generated by transformation semigroups of degree at most $2$. We recall the generalization of this work to transformation semigroups of degree at most $k$ and their faithful representations on simplicial complexes of dimension $k-1$.

In the rest of the paper we write ``complexity'' for Krohn-Rhodes complexity. For background in complexity theory see \cite{Arbib, qtheory}. Throughout the paper for all $n \ge 1$, $\Pn$ denotes the path graph with vertex set $V = \{1, \dots, n\}$ and edge set $E = \{\{i, i+1\} \mid 1 \le i < n\}$.

\section{Preliminaries}

\subsection{Semigroups and Transformation Semigroups}
We assume the reader is familiar with the basic theory of finite semigroups and monoids. A \emph{transformation semigroup} $(Q, S)$ consists of a finite set $Q$ and a subsemigroup $S$ of the partial transformation monoid $\mathcal{PT}(Q)$ acting on the right of $Q$. The action of an element $s \in S$ on a state $q \in Q$ is denoted by $qs$ if $q \in \operatorname{Dom}(s)$ and undefined otherwise. 

Two transformation semigroups $(Q_1, S_1)$ and $(Q_2, S_2)$ are \emph{isomorphic} if there exists a bijection $\phi \colon Q_1 \to Q_2$ and a semigroup isomorphism $\psi \colon S_1 \to S_2$ such that $(qs)\phi = (q\phi)(s\psi)$ for all $q \in Q_1$ and $s \in S_1$, where defined. 

We freely use standard concepts from semigroup theory, including Green's relations ($\mathcal{R}$, $\mathcal{L}$, $\mathcal{J}$, $\mathcal{H}$, $\mathcal{D}$), ideals, idempotents, and regular elements, as found in standard introductory texts such as \cite{Arbib, CP, GM, howiebook, Lallement, qtheory}.

\subsection{The Krohn-Rhodes Theorem and Complexity}
The Krohn-Rhodes Decomposition Theorem is a foundational result in algebraic automata theory \cite{Arbib, Lallement, qtheory}. It states that every finite transformation semigroup $(Q, S)$ divides a wreath product of finite groups and finite aperiodic semigroups. A semigroup is aperiodic if it contains no non-trivial subgroups. 

The \emph{Krohn-Rhodes complexity} (or simply \emph{complexity}), denoted $Sc$, is the minimum number of groups required in such a wreath product decomposition to faithfully cover $S$. By definition:
\begin{enumerate}
    \item $Sc = 0$ if and only if $S$ is aperiodic.
    \item $Sc \le n$ if $S$ divides a wreath product of the form $A_n \wr G_n \wr \dots \wr A_1 \wr G_1 \wr A_0$, where the $G_i$ are finite groups and the $A_i$ are finite aperiodic semigroups.
\end{enumerate}

Recently, Margolis, Rhodes and Schilling \cite{complexity1, complexityn} proved that the complexity $Sc$ is computable. The semigroups considered in this paper will have complexity at most 2. We will look for combinatorial, geometric and topological conditions on a transformation semigroup in our class that helps distinguish between those of complexity 1 and those of complexity 2.

\subsection{Monoids of Continuous Maps on Incidence Structures}
The algebraic study of continuous maps on discrete combinatorial structures was initiated by Margolis \cite{Tilsonnumber}. An \emph{incidence structure} is a pair $\mathcal{D} = (V, \mathcal{B})$ where $V$ is a set of points and $\mathcal{B}$ is a collection of non-empty subsets of $V$ called blocks. Important examples include BIBDs (Balanced Incomplete Block Designs), PBDs (Pairwise Balanced Designs), Affine Geometries (where the blocks are points, lines, planes, etc.), projective geometries, simplicial complexes (where simplices form the blocks) and more.

A partial function $f \colon V \rightarrow V$ is said to be \emph{continuous} if the full inverse image of every block $B \in \mathcal{B}$ is either a block or the empty set; that is, $B f^{-1} \in \mathcal{B} \cup \{\emptyset\}$. 

The collection of all such continuous maps forms a submonoid of the partial transformation monoid on $V$. Dinitz and Margolis \cite{bibdtrans, bibdcont, projcont} classified the continuous maps for finite projective spaces and various classes of BIBDs, demonstrating that these monoids encode geometry beyond automorphism groups, which form the group of units of the monoid of continuous functions on a design. 

In \cite{AmigoWilson}, Margolis, Rhodes, and Silva showed that continuous maps on open subsets of pairwise balanced designs are precisely the morphisms used by Wilson \cite{Wilson1, Wilson2, Wilson3} in his fundamental result on the existence of designs. In this connection, the monoid of continuous morphisms on the open sets of a PBD $\mathcal{D}$ is called the \emph{Wilson Monoid} of $\mathcal{D}$. It is an algebraic invariant capturing the combinatorics of such structures. It gives a deep connection between PBDs and Truncated Boolean Representable Simplicial Complexes \cite{Stu9}. The paper \cite{AmigoWilson} surveys Wilson monoids for Steiner triple systems of size up to 19.

\subsection{Transformation Semigroups of Degree 2}
The continuous maps defined on 1-dimensional simplicial complexes (graphs) provide the geometric foundation for studying transformation semigroups with fiber size at most 2 . Let $(Q, S)$ be a finite partial transformation semigroup. We say that $(Q, S)$ has \emph{degree at most 2} if, for every $s \in S$ and every $q \in Q$, the cardinality of the fiber satisfies $|q s^{-1}| \le 2$.

In \cite{cremona} and \cite{deg2part2}, the authors established a connection between degree 2 transformation semigroups and the topology of graphs. If $\Gamma = (V, E)$ is a finite simple graph, we view it as a simplicial complex where the simplices are the vertices and edges. The Margolis-Rhodes monoid $\operatorname{MR}(\Gamma)$ is the monoid of all continuous partial functions on $V$. Thus these are the partial functions such that the inverse image of every vertex or edge is either a vertex, an edge, or empty. The Margolis-Rhodes Representation Theorem proves that a finite transformation semigroup has degree of at most $2$ if and only if it can be faithfully represented as a subsemigroup of $\operatorname{MR}(\Gamma)$ for some graph $\Gamma$.

More generally, a transformation semigroup $(Q,S)$ has degree at most $k$ (we write $(X,S)d \leq k$) if for each $x \in X, s \in S$, the fiber $x s^{-1}$ satisfies $|x s^{-1}| \leq k$. Margolis proved in \cite{Tilsonnumber} that a transformation semigroup has degree at most $k$ if and only if it is embeddable into a monoid of continuous functions on a simplicial complex of dimension at most $k-1$. The Margolis-Rhodes Representation Theorem is the case $k=2$ of this case.

The main result of \cite{Tilsonnumber} shows that if $(Q,S)$ is a transformation semigroup of degree $k$, then the complexity of $S$ is bounded above by $k$. That is, degree is an upper bound to complexity, $(Q,S)c \leq (Q,S)d$.

In this paper we study the case of transformation semigroups of degree 2 as continuous maps on graphs. Thus a transformation semigroup $(Q,S)$ of degree at most 2 is either aperiodic, has complexity 1 or has complexity 2. We use algebraic, combinatorial, geometric and topological tools to distinguish between the transformation semigroups of degree 2 with complexity 1 and those with complexity 2.

\section{The Margolis-Rhodes Monoids and Their Subsemigroups}

\subsection{Topological Continuity on Graphs}
Let $\Gamma = (V, E)$ be a finite simple graph. We view $\Gamma$ as a $1$-dimensional simplicial complex whose set of simplices $\Delta(\Gamma)$ consists of all singletons $\{v\}$ for $v \in V$ and all pairs $\{u, v\}$ such that $\{u, v\} \in E$. For any partial transformation $f \colon V \rightarrow V$, we denote its domain by $\operatorname{Dom}(f) \subseteq V$ and its image by $\operatorname{Im}(f) = (\operatorname{Dom}(f))f \subseteq V$. For any subset $U \subseteq V$, the full inverse image under $f$ is defined as:
\[
U f^{-1} = \{ v \in \operatorname{Dom}(f) \mid vf \in U \}.
\]
Following the general formulation of continuous maps on incidence structures, we formally define the full Margolis-Rhodes monoid.

\begin{definition}
The \emph{Margolis-Rhodes monoid} $\operatorname{MR}(\Gamma)$ is the set of all partial transformations $f \colon V \rightarrow V$ such that for every simplex $\sigma \in \Delta(\Gamma)$, the full preimage $\sigma f^{-1}$ belongs to $\Delta(\Gamma) \cup \{\emptyset\}$.
\end{definition}

Composition of elements in $\operatorname{MR}(\Gamma)$ is the standard composition of partial functions from left to right: for $f, g \in \operatorname{MR}(\Gamma)$, the domain of $fg$ is $\operatorname{Dom}(fg) = \{v \in \operatorname{Dom}(f) \mid vf \in \operatorname{Dom}(g)\}$, and the action yields $v(fg) = (vf)g$. It is immediate that if $f$ and $g$ are continuous, their composition $fg$ is also continuous, making $\operatorname{MR}(\Gamma)$ a monoid with the identity map $1_V \colon V \rightarrow V$ serving as the identity element.

\subsection{Strict Continuous Maps on a Graph}

We now define a subsemigroup $\StG{\Gamma}$ called the {\em semigroup of strict continuous functions} on  $\Gamma$ defined by a stricter topological requirement on vertex fibers.

\begin{definition}
The subsemigroup $\StG{\Gamma} \subseteq \operatorname{MR}(\Gamma)$ consists of all continuous partial transformations $f \in \operatorname{MR}(\Gamma)$ such that for every vertex $v \in V$, the fiber $v f^{-1}$ is strictly a single edge or empty. That is, $v f^{-1} \in E \cup \{\emptyset\}$.
\end{definition}

This requirement leads to an elegant combinatorial characterization of the elements of $\StG{\Gamma}$ in terms of the underlying graph geometry. Recall that a \emph{matching} in $\Gamma$ is a set of pairwise vertex-disjoint edges, and an \emph{independent set} (sometimes called an anti-clique) is a set of pairwise non-adjacent vertices.

\begin{lemma}\label{strict}
A partial transformation $f$ belongs to $\StG{\Gamma}$ if and only if:
\begin{enumerate}
    \item $\operatorname{Dom}(f)$ is the union of vertices of a matching $M_f = \{e_1, e_2, \dots, e_k\}$ in $\Gamma$.
    \item $\operatorname{Im}(f)$ is an independent set $I_f = \{v_1, v_2, \dots, v_k\}$ in $\Gamma$.
    \item Both endpoints of each edge $e_i \in M_f$ are mapped by $f$ to the unique vertex $v_i \in I_f$, i.e., $u f = v_i$ for all $u \in e_i$.
\end{enumerate}
\end{lemma}

\begin{proof}
Suppose $f \in \StG{\Gamma}$. By definition, for every $v \in \operatorname{Im}(f)$, $v f^{-1}$ is a single edge. Since $f$ is a well-defined function, the fibers $v f^{-1}$ for distinct vertices in $\operatorname{Im}(f)$ must be vertex-disjoint. Thus, the collection of edges $M_f = \{v f^{-1} \mid v \in \operatorname{Im}(f)\}$ forms a matching, and $\operatorname{Dom}(f) = \bigcup_{e \in M_f} e$.

Next, we show that $\operatorname{Im}(f)$ is an independent set. Suppose for a contradiction that there exist distinct $v, w \in \operatorname{Im}(f)$ such that the edge $e = \{v, w\} \in E$. By the definition of $\operatorname{MR}(\Gamma)$, continuity requires that the preimage $e f^{-1} = v f^{-1} \cup w f^{-1}$ must be an element of $\Delta(\Gamma) \cup \{\emptyset\}$. However, since $f \in \StG{\Gamma}$, the fibers $v f^{-1} = e_i$ and $w f^{-1} = e_j$ are two distinct, disjoint edges. Their union $e_i \cup e_j$ contains exactly 4 distinct vertices, which can neither be a single vertex nor a single edge. This contradiction proves that $\operatorname{Im}(f)$ is an independent set.

Conversely, if $f$ satisfies conditions (1)--(3), then for any vertex $v \in V$, $v f^{-1}$ is either an edge in $M_f$ (if $v \in I_f$) or empty. For any edge $e = \{v, w\} \in E$, since $\operatorname{Im}(f) = I_f$ is an independent set, at most one of $v$ or $w$ can belong to $\operatorname{Im}(f)$. Thus, the inverse image of any edge $\{v, w\}$ is at most the inverse image of a single vertex, which is an edge in $M_f$ or empty. In all cases, preimages of simplices map to simplices or the empty set. Hence, $f \in \operatorname{MR}(\Gamma)$ and satisfies the defining condition of $\StG{\Gamma}$.
\end{proof}

\begin{corollary}

The maximum rank of an element $f \in \operatorname{St}(\Gamma)$ is $\min\{\alpha(\Gamma), \mu(\Gamma)\}$ where $\mu(\Gamma)$ is the largest size of a matching of $\Gamma$ and $\alpha(\Gamma)$ is the largest size of an independent set in $\Gamma)$. In a bipartite graph, the maximum rank of an element of $\operatorname{St}(\Gamma)$ is $\mu(\Gamma)$.

\end{corollary}

\begin{proof}

The  first statement follows from Lemma \ref{strict}. In a bipartite graph it is well known that $\mu(\Gamma) \le \alpha(\Gamma)$.
    
\end{proof}
\subsection{The Submonoid $Inj(\Gamma)$}

On the opposite end of fiber restrictions, we can enforce that every partial map is one-to-one.

\begin{definition}
The submonoid $\Inj(\Gamma) \subseteq \operatorname{MR}(\Gamma)$ consists of all injective continuous partial maps on $\Gamma$. 
\end{definition}

For any subset $U \subseteq V$, we denote by $\Gamma[U]$ the induced subgraph of $\Gamma$ on the vertex set $U$. Injective continuous maps have a completely natural description via induced subgraph embeddings. Recall that a graph morphism, that is a function between graphs that sends an edge to an edge, is an embedding if it is 1-1 on vertices and edges.

\begin{lemma}\label{Inj}
A partial transformation $f$ belongs to $\Inj(\Gamma)$ if and only if $f$ is a bijection from $\operatorname{Dom}(f)$ to $\operatorname{Im}(f)$, and the inverse mapping $f^{-1}$ induces a graph embedding from $\Gamma[\operatorname{Im}(f)]$ to $\Gamma[\operatorname{Dom}(f)]$.
\end{lemma}

\begin{proof}
Let $f \in \Inj(\Gamma)$. Since $f$ is injective, so is $f^{-1}$, that is, for every $v \in \operatorname{Im}(f)$, the fiber $v f^{-1}$ is a single vertex. If $w \in \operatorname{Im}(f)$ and $\{v, w\} \in E$, continuity forces $\{v, w\} f^{-1} = v f^{-1} \cup w f^{-1}$ to be a simplex. Since $f$ is injective, these two preimages are disjoint vertices, so their union must form an edge in $\Gamma[\operatorname{Dom}(f)]$. Therefore, $f^{-1}$ is a graph embedding.

Conversely, assume that $f$ is a partial function such that $f^{-1}$ is a graph embedding from $\Gamma[\operatorname{Im}(f)]$ to $\Gamma[\operatorname{Dom}(f)]$. Let $v$ be a vertex. If $v$ is not in $\operatorname{Im}(f)$ then $vf^{-1}$ is empty. If $v$ is in $\operatorname{Im}(f)$ then $vf^{-1}$ is a vertex since $f^{-1}$ is injective. Now let $e=\{v,w\}$ be an edge of $\Gamma$. If neither vertex of $e$ is in $\operatorname{Im}(f)$, then $ef^{-1}$ is empty. If one vertex of $e$ is in $\operatorname{Im}(f)$ then $ef^{-1}$ is a vertex by injectivity. If both vertices of $e$ are in $\operatorname{Im}(f)$ then $ef^{-1}$ is an edge of $\Gamma$ since $f^{-1}$ is a graph embedding. Therefore $f \in \operatorname{Inj}(\Gamma)$.

\end{proof}

We characterize the regular elements of $\operatorname{Inj}(\Gamma)$. It is well known that if $S$ is a semigroup of partial 1-1 functions, then the set of regular elements $\operatorname{Reg}(S)$ is an inverse semigroup.

\begin{corollary}

An element $f$ of $\operatorname{Inj}(\Gamma)$ belongs to $\operatorname{Reg}(\operatorname{Inj}(\Gamma))$ if and only if $f$ is an isomorphism between $\Gamma[\operatorname{Dom}(f)]$ and $\Gamma[\operatorname{Im}(f)]$.

\end{corollary}

\begin{proof}

It follows from Lemma \ref{Inj} that if $f$ is an isomorphism between $\Gamma[\operatorname{Dom}(f)]$ and $\Gamma[\operatorname{Im}(f)]$, then both $f$ and $f^{-1}$ belong to $\operatorname{Inj}(\Gamma)$.  Since $ff^{-1}f=f$, $f$ is a regular element of $\operatorname{Inj}(\Gamma)$.

Conversely, if $f$ is a regular element of $\operatorname{Inj}(\Gamma)$, then $f^{-1} \in \operatorname{Inj}(\Gamma)$. This is because  $\operatorname{Inj}(\Gamma)$ is a semigroup of partial 1-1 functions and an element is regular if and only if its inverse also belongs to $\operatorname{Inj}(f)$ by standard inverse semigroup theory \cite{Lawson}. Therefore, by Lemma \ref{Inj}  $f^{-1}$ is a graph embedding between   $\Gamma[\operatorname{Im}(f)]$ and $\Gamma[\operatorname{Dom}(f)]$  and similarly $f$ is a graph embedding between $\Gamma[\operatorname{Dom}(f)]$ and $\Gamma[\operatorname{Im}(f)]$. Since $\Gamma$ is a finite graph, $f$ is an isomorphism between $\Gamma[\operatorname{Dom}(f)]$ and $\Gamma[\operatorname{Im}(f)]$.

\end{proof}

\subsection{A Complete Characterization of Continuous Functions}
We establish an explicit characterization of the continuous functions in the Margolis-Rhodes monoid. This formulation translates the  definition into a structural decomposition of any continuous function into singular and injective components.

For any partial function $f \in \operatorname{MR}(\Gamma)$, let the singular domain of $f$ be the set $\mathcal{M}(f) = \{v \in \operatorname{Dom}(f) \mid \exists w \neq v, wf = vf\}$. That is, $\mathcal{M}(f)$ is the union of all fibers of $f$ of size 2. We define the \emph{singular part} of $f$, denoted $\operatorname{Sing}{f}$, to be the restriction of $f$ to $\mathcal{M}(f)$. We define the \emph{injective part} of $f$, denoted $\operatorname{Inj}(f)$, to be the restriction of $f$ to $\operatorname{Dom}(f) \setminus \mathcal{M}(f)$. Any continuous function can therefore be written as the disjoint join $f = \operatorname{Sing}(f) \vee \operatorname{Inj}(f)$. 

The behavior of these two components is characterized by the following lemma. Recall that a bijective graph morphism is a bijection that preserves edges but is not necessarily a full graph isomorphism. The Lemma shows that $\operatorname{Sing}(f) \in \StG\Gamma$ and $\operatorname{Inj}(f) \in Inj(\Gamma)$.

\begin{lemma}[Lemma 2.6 of \cite{deg2part2}]\label{lem:lemma26}
Let $\Gamma = (V, E)$ be a finite simple graph.
\begin{enumerate}
    \item Let $f \in \operatorname{MR}(\Gamma)$ be a continuous partial function. Then the partition $\ker(\operatorname{Sing}(f))$ is a matching of $\Gamma$, and the image $\operatorname{Im}(\operatorname{Sing}(f))$ is an independent set in $\Gamma$. Moreover, $\operatorname{Im}(\operatorname{Sing}(f)) \cap \operatorname{Im}(\operatorname{Inj}(f)) = \emptyset$, the inverse partial function $\operatorname{Inj}(f)^{-1}$ is a bijective graph morphism onto its image, and there are no edges between a vertex in $\operatorname{Im}(\operatorname{Sing}(f))$ and a vertex in $\operatorname{Im}(\operatorname{Inj}(f))$.
    \item Conversely, let $g$ be a partial function such that $\operatorname{Ker}(g)$ is a matching in $\Gamma$ and $\operatorname{Im}(g)$ is an independent set in $\Gamma$. Let $h$ be a partial bijection such that $h^{-1}$ is a bijective graph morphism onto its image. Assume further that $\operatorname{Dom}(g) \cap \operatorname{Dom}(h) = \operatorname{Im}(g) \cap \operatorname{Im}(h) = \emptyset$ and that there are no edges between $\operatorname{Im}(g)$ and $\operatorname{Im}(h)$. Then $f = g \vee h$ is continuous, and $g = \operatorname{Sing}(f)$ and $h = \operatorname{Inj}(f)$.
\end{enumerate}
\end{lemma}

This lemma provides an extremely powerful practical tool for determining membership in $\operatorname{MR}(\Gamma)$. It explicitly bounds the degree of the transformation semigroup $(V, \operatorname{MR}(\Gamma))$ by 2, as the size of any vertex fiber cannot exceed 2 (since the singular kernel strictly forms a matching of size-2 edges). It also ensures that the components do not topologically interfere with each other, separating structural collapse from embeddings.

\section{Green's Relations and Ideal Structure}

\subsection{\texorpdfstring{$\mathcal{J}$}{J}-Equivalence Criteria in \texorpdfstring{$\operatorname{MR}(\Gamma)$}{MR(Gamma)}}
This section provides a classification of $\mathcal{J}$-classes for regular elements within the Margolis-Rhodes monoid $\operatorname{MR}(\Gamma)$ for an arbitrary finite graph $\Gamma = (V, E)$. It is well known that for regular elements in any transformation semigroup $(Q,S)$, $\mathcal{R}$ equivalence and $\mathcal{L}$ equivalence are the same in $S$ and $PT(Q)$. That is, for regular elements $f,g \in \operatorname{MR}(\Gamma)$, $f\mathcal{R}g$ if and only if $\operatorname{Ker}(f)=\operatorname{Ker}(g)$ and $f\mathcal{L}g$ if and only if $\operatorname{Im}(f)=\operatorname{Im}(g)$. Here $\operatorname{Ker}(f)$ is the equivalence relation on $\operatorname{Dom}(f)$ whose classes are the fibers of $f$.

\subsubsection{\texorpdfstring{$\mathcal{J}$}{J}-Equivalence Criterion for Idempotents (Regular Elements)}

It is obvious that if $\Gamma=(V,E)$ is a graph and $X \subseteq V$, then the partial identity $e_{X}$ with domain $X$ is in $\operatorname{MR}(\Gamma)$. It follows that every idempotent in $\operatorname{MR}(\Gamma)$ is $\mathcal{L}$-equivalent to $e_{\operatorname{Im}(e)}$. We record this in the next Lemma.

\begin{lemma}\label{foundational}
Every idempotent $e \in \operatorname{MR}(\Gamma)$ is $\mathcal{L}$-equivalent to the partial identity map on its image, $e_{\operatorname{Im}(e)}$.
\end{lemma}

\begin{proof}

It is well known that in any subsemigroup of $PT_{n}$, two idempotents $e,f$ are $\mathcal{L}$-equivalent if and only if
$\operatorname{Im}(e)= \operatorname{Im}(f)$. Since $e_{\operatorname{Im}(e)} \in \operatorname{MR}(\Gamma)$, the result follows.





\end{proof}

\begin{theorem} \label{JMR}
Let $\Gamma = (V, E)$ be an arbitrary graph. Two idempotents $e, f \in \operatorname{MR}(\Gamma)$ are $\mathcal{J}$-equivalent if and only if the induced subgraphs on their images are isomorphic: $\Gamma[\operatorname{Im}(e)] \cong \Gamma[\operatorname{Im}(f)]$.
\end{theorem}

\begin{proof}

Assume $e \mathrel{\mathcal{J}} f$. By Lemma \ref{foundational}, this implies that  $e_{\operatorname{Im}(e)} \mathrel{\mathcal{J}} e_{\operatorname{Im}(f)}$.

By definition of $\mathcal{J}$-equivalence, there exist continuous partial maps $A, B \in \operatorname{MR}(\Gamma)$ such that:
\[ e_{\operatorname{Im}(e)} = A \cdot e_{\operatorname{Im}(f)} \cdot B \]
Let us evaluate the restriction of $A$ to the domain $\operatorname{Im}(e)$. For any vertex $x \in \operatorname{Im}(e)$, we have:
\[ x = x e_{\operatorname{Im}(e)} = ((x A) e_{\operatorname{Im}(f)}) B \]
For this composition to be defined and equal to $x$:
\begin{enumerate}
    \item $x A$ must be defined and must land inside the domain of $e_{\operatorname{Im}(f)}$, which is exactly $\operatorname{Im}(f)$. Thus, $A$ maps $\operatorname{Im}(e)$ into $\operatorname{Im}(f)$.
    \item $B$ acts as a right inverse to $A$ on this subset ($(x A) B = x$), which forces $A$ restricted to $\operatorname{Im}(e)$ to be an injective vertex map.
\end{enumerate}

Now let us check if $A$ preserves edges. Let $\{x, y\}$ be an edge in the induced subgraph $\Gamma[\operatorname{Im}(e)]$. This means $\{x, y\} \in E$ and $x, y \in \operatorname{Im}(e)$.

Because $B \in \operatorname{MR}(\Gamma)$ is a continuous partial map, the definition of continuity states that the inverse image of any cell must be a cell (or empty). Consider the edge cell $\sigma = \{x, y\}$. Its preimage under $B$ must contain the vertices $x A$ and $y A$ because $(x A) B = x$ and $(y A) B = y$. Since $A$ is injective, $x A \neq y A$. The only cells in a graph complex containing two distinct vertices are edges. Therefore, $\{x, y\} B^{-1}$ must be a single valid edge in $\Gamma$ connecting $x A$ and $y A$. Since both endpoints land in $\operatorname{Im}(f)$, $\{x A, y A\}$ is a valid edge in the induced subgraph $\Gamma[\operatorname{Im}(f)]$.

Thus, $A$ restricted to $\operatorname{Im}(e)$ is an edge-preserving injective graph homomorphism. By symmetry, because $e_{\operatorname{Im}(f)} \mathrel{\mathcal{J}} e_{\operatorname{Im}(e)}$, there exists a continuous partial map creating an edge-preserving injection from $\Gamma[\operatorname{Im}(f)]$ into $\Gamma[\operatorname{Im}(e)]$.

For finite graphs, the existence of edge-preserving injections in both directions guarantees that the number of vertices and edges are identical, meaning the mapping is a bijection that perfectly preserves adjacency. Thus, $\Gamma[\operatorname{Im}(e)] \cong \Gamma[\operatorname{Im}(f)]$.

Assume conversely that the induced subgraphs are isomorphic, and let $\psi \colon \Gamma[\operatorname{Im}(e)] \to \Gamma[\operatorname{Im}(f)]$ be the graph isomorphism.

We define a partial function $A \colon V \to V$ defined on the domain $\operatorname{Dom}(A) = \operatorname{Im}(e)$, where $x A = x \psi$. We first prove that this partial function is a member of $\operatorname{MR}(\Gamma)$.

Let us check the continuity of $A$ against an arbitrary cell $\sigma$ of $\Gamma$:
\begin{itemize}
    \item \textbf{Case 1:} $\sigma = \{v\}$ is a vertex cell. The preimage $v A^{-1}$ is the set of vertices in $\operatorname{Im}(e)$ mapping to $v$. Since $\psi$ is a bijection between the two subsets, if $v \in \operatorname{Im}(f)$, the preimage is a single vertex (a valid cell). If $v \notin \operatorname{Im}(f)$, the preimage is empty.
    \item \textbf{Case 2:} $\sigma = \{u, v\}$ is an edge cell. The preimage is $\{u, v\} A^{-1} = (\{u, v\} \cap \operatorname{Im}(f)) \psi^{-1}$.
    \begin{itemize}
        \item If $\{u, v\} \cap \operatorname{Im}(f) = \emptyset$, the preimage is empty.
        \item If $\{u, v\} \cap \operatorname{Im}(f) = \{u\}$ (only one vertex is in the image), the preimage is a single vertex $u \psi^{-1}$.
        \item If both $u, v \in \operatorname{Im}(f)$, then $\{u, v\}$ is an edge belonging to the induced subgraph $\Gamma[\operatorname{Im}(f)]$. Since $\psi$ is a graph isomorphism, the preimage of an edge under an isomorphism is guaranteed to be a single edge in $\Gamma[\operatorname{Im}(e)]$, which is a valid edge cell in $\Gamma$.
    \end{itemize}
\end{itemize}

Because the preimage of every cell is a valid cell or empty, the partial isomorphism $A$ is continuous, so $A \in \operatorname{MR}(\Gamma)$. By the exact same logic, its inverse partial map $B$ defined strictly on $\operatorname{Im}(f)$ where $y B = y \psi^{-1}$ is also a valid element of $\operatorname{MR}(\Gamma)$.

Now, let us look at the composition $A \cdot e_{\operatorname{Im}(f)} \cdot B$:
\begin{itemize}
    \item Its domain is restricted by $A$ to exactly $\operatorname{Im}(e)$.
    \item For any $x \in \operatorname{Im}(e)$, $x A = x \psi \in \operatorname{Im}(f)$.
    \item The partial identity $e_{\operatorname{Im}(f)}$ fixes $x \psi$, and then $(x \psi) B = (x \psi) \psi^{-1} = x$.
\end{itemize}
This composition is a partial map defined exactly on $\operatorname{Im}(e)$ acting as the identity, which is precisely the definition of $e_{\operatorname{Im}(e)}$. Therefore:
\[ e_{\operatorname{Im}(e)} = A \cdot e_{\operatorname{Im}(f)} \cdot B \implies e_{\operatorname{Im}(e)} \in \operatorname{MR}(\Gamma) e_{\operatorname{Im}(f)} \operatorname{MR}(\Gamma) \]
By symmetry, using $B$ and $A$ in reverse, we get $e_{\operatorname{Im}(f)} \in \operatorname{MR}(\Gamma) e_{\operatorname{Im}(e)} \operatorname{MR}(\Gamma)$. Thus, $e_{\operatorname{Im}(e)} \mathrel{\mathcal{J}} e_{\operatorname{Im}(f)}$. By Lemma \ref{foundational}, this directly proves that $e \mathrel{\mathcal{J}} f$.
\end{proof}

\subsection{The Poset of Regular J-Classes of {\texorpdfstring{$\operatorname{MR}(\Gamma)$}{MR(Gamma)}}}

We analyze the partial order of the regular $\mathcal{J}$-classes of $\operatorname{MR}(\Gamma)$. For $f \in \operatorname{MR}(\Gamma)$, recall that $\Gamma[\operatorname{Im}(f)]$ is the induced subgraph on the image of $f$.

\begin{theorem}
The regular $\mathcal{J}$-classes of $\operatorname{MR}(\Gamma)$ are in one-to-one correspondence with the isomorphism classes of induced subgraphs of $\Gamma$. The poset of regular $\mathcal{J}$-classes of $\operatorname{MR}(\Gamma)$ is isomorphic to the poset of induced subgraphs of $\Gamma$ under inclusion.
\end{theorem}

\begin{proof}

By Lemma \ref{foundational} and Theorem \ref{JMR} it follows that the $\mathcal{J}$-class of $f$ is determined by the isomorphism class of $\Gamma[\operatorname{Im}(f)]$.


Let $f, g \in \operatorname{MR}(\Gamma)$ be regular elements of $\operatorname{MR}(\Gamma)$ with $f \le_{\mathcal{J}} g$. Let $e,k$ be idempotents with $k\mathcal{L}f, e\mathcal{L}g$. Then there are $x,y$ such that $k=xey$. Consider $k'=ykxe$. Then $k'$ is an idempotent with $xek'y=k$ so that $k\mathcal{J}k'$ and $k' \le_{\mathcal{L}} e$. 
In particular, $\Gamma[\operatorname{Im}(k)]$ is an induced subgraph of $\Gamma[\operatorname{Im}(e)]$. Therefore, $\Gamma[\operatorname{\operatorname{Im}(f)}]$ is isomorphic to an induced subgraph of $\Gamma[\operatorname{Im}(g)]$.

Conversely, assume that $\Gamma[\operatorname{Im}(f)]$ is isomorphic to an induced subgraph of $\Gamma[\operatorname{Im}(g)]$ for regular elements $f$ and $g$. Without loss of generality we can assume that $f$ and $g$ are idempotents. Then $f \le_{\mathcal{L}} g$ and thus $f\le_{\mathcal{J}} g$. It follows immediately that the poset of regular $\mathcal{J}$-classes of $\operatorname{MR}(\Gamma)$ under the natural ideal order is isomorphic to the poset of induced subgraphs of $\Gamma$ ordered by the embedding relation.
\end{proof}

Regularity of the monoid $\operatorname{MR}(\Gamma)$ is rare.

\begin{proposition}
The Margolis-Rhodes monoid $\operatorname{MR}(\Gamma)$ is a regular monoid if and only if $\Gamma$ is an empty graph or a complete graph $K_n$.
\end{proposition} \label{RegJMR}

\begin{proof}

Assume $\Gamma$ is the empty graph on an $n$-set. Then clearly, $\operatorname{MR}(\Gamma)$ is isomorphic to the symmetric inverse monoid $SIM(n)$ and is thus regular. If $\Gamma=K_n$ it is easy to see that any member of $SIM(n)$ belongs to $\operatorname{MR}(\Gamma)$. If $f \in MR(K_{n})$ is not 1-1, it follows from Lemma \ref{lem:lemma26} that since in $K_n$, the largest independent set of $K_n$ has size 1, that $f$ has domain $e=\{x,y\}$ for some edge of $K_n$ and $xf= yf=v$ for some vertex of $K_n$. Let $g$ be the element of $SIM(n)$ of rank 1 with $vg=x$. Then $fgf=f$ and it follows that $MR(K_n)$ is a regular monoid.

Assume that $\Gamma$ is not the empty graph nor the complete graph on some set. Then there is an edge $\{v,w\} \in E(\Gamma)$ and a pair of vertices $x,y \in V(\Gamma)$ that does not form an edge. Let $f$ have domain $\{v,w\}$ with $vf=x,wf=y$. Then $f$ is a member of $\operatorname{MR}(\Gamma)$ as can be easily verified. Let $g \in \operatorname{MR}(\Gamma)$. If $\{x,y\}g$ has cardinality at most 1, then clearly $fgf \ne f$. If $\{x,y\}g$ has cardinality 2, then $\{x,y\}g$ is not an edge by continuity. Since $\{v,w\}$ is an edge, it follows that $fgf \ne f$ in this case as well. This proves the proposition.

\end{proof}

We consider the case of $St(\Gamma)$ later in the paper. 

\section{Maximal Submonoids and Maximal Subgroups of $\operatorname{MR}(\Gamma)$}

In this section we determine the maximal submonoids and maximal subgroups of $\operatorname{MR}(\Gamma)$.

\begin{theorem}\label{sec:local_monoids}
Let $\Gamma = (V,E)$ be a finite simple graph, and let $e \in \operatorname{MR}(\Gamma)$ be an idempotent. Then the local submonoid $e \operatorname{MR}(\Gamma) e$ is isomorphic to $MR(\Gamma[\operatorname{Im}(e)])$, where $\Gamma[\operatorname{Im}(e)]$ denotes the induced subgraph of $\Gamma$ on the vertex set $\operatorname{Im}(e)$.
\end{theorem}

\begin{proof}

By Lemma \ref{foundational} we can assume that $e$ is a partial identity on some subset $\operatorname{Im}(e)$ of $V$. Therefore if $f \in e\operatorname{MR}(\Gamma)e$, then both $\operatorname{Dom}(f)$ and $\operatorname{Im}(f)$ are contained in $\Gamma[\operatorname{Im}(e)]$. Therefore $f$ can be considered to be a partial function $V(\Gamma[\operatorname{Im}(e)] \rightarrow V(\Gamma[\operatorname{Im}(e))$. Clearly, $f$ is continuous. Therefore, $e\operatorname{MR}(\Gamma)e$ is a submonoid of  
$\Gamma[\operatorname{Im}(e)]$.

Conversely if $f \in \Gamma[\operatorname{Im}(e)]$, then we consider $f:V\rightarrow V$ to be a partial function by having it be undefined on $V \setminus \operatorname{Im}(e)$. It is straightforward to check that this puts $f$ in $e\operatorname{MR}(\Gamma)e$.

\end{proof}

\begin{corollary}
Let $\Gamma = (V,E)$ be a finite simple graph, and let $e \in \operatorname{MR}(\Gamma)$ be an idempotent. Then the maximal subgroup containing $e$ in its $\mathcal{J}$-class is isomorphic to the automorphism group of the induced subgraph on the image of $e$:
\[
H_e \cong \text{Aut}(\Gamma[\text{Im}(e)])
\]
\end{corollary}

\begin{proof}
By a classical result of Green's relations in semigroup theory, the maximal subgroup of a semigroup containing an idempotent $e$ is precisely the group of units (the invertible elements) of the submonoid $e \operatorname{MR}(\Gamma) e$. 

From the isomorphism established in Theorem~\ref{sec:local_monoids}, we have $e \operatorname{MR}(\Gamma) e \cong MR(\Gamma[\operatorname{Im}(e)])$. Under this isomorphism, the group of units of $e \operatorname{MR}(\Gamma) e$ maps bijectively to the group of units of $MR(\Gamma[\operatorname{Im}(e)])$. An element in  $MR(\Gamma[\operatorname{}{Im}(e)])$ is invertible if and only if it is a total, bijective, structure-preserving map on the underlying vertex set $\operatorname{Im}(e)$. By definition, the group of such invertible continuous total bijections on an induced subgraph is precisely its graph automorphism group, $\operatorname{Aut}(\Gamma[\operatorname{Im}(e)])$. Therefore, $H_e \cong \operatorname{Aut}(\Gamma[\text{Im}(e)])$.
\end{proof}

\section{The Regular \texorpdfstring{$\mathcal{J}$}{J}-Class Poset of \texorpdfstring{$MR(P_n)$}{MR(Pn)} and the Graph Partition Lattice}
\label{sec:mpn_partitions}

To illustrate Theorem \ref{RegJMR}, we study the ideal structure of the Margolis-Rhodes monoid $MR(P_n)$ for a path graph $P_n$ containing $n$ vertices labeled $1, 2, \dots, n$ and $n-1$ edges of the form $\{j, j+1\}$. A similar analysis can be done for $MR(C_{n})$ where $C_n$ is the cycle of size $n$.

\subsection{Isomorphism to the Integer Partition Lattice}

Recall from Theorem~5.2 that two regular elements $f, g \in \operatorname{MR}(\Gamma)$ belong to the same $\mathcal{J}$-class if and only if their induced image subgraphs are isomorphic:
\[
f \mathcal{J} g \iff \Gamma[\text{Im}(f)] \cong \Gamma[\text{Im}(g)]
\]
For a path graph $\Gamma = P_n$ this poset is the lattice of integer partitions of $n+1$.

\begin{theorem}
The poset of regular $\mathcal{J}$-classes of $MR(P_n)$ ordered by ideal inclusion is isomorphic to the lattice of integer partitions of $n+1$ ordered by refinement.
\end{theorem}

\begin{proof}
Let $U \subseteq V(P_n)$ be an arbitrary vertex subset. Because $P_n$ is a line graph, the induced subgraph $P_n[U]$ must be either $P_k$ for some $k \le n$ or a disjoint union of smaller path graphs separated by at least one missing vertex gap:
\[
P_n[U] \cong P_{\lambda_1} \cup P_{\lambda_2} \cup \dots \cup P_{\lambda_k}
\]
where $\lambda_1 \ge \lambda_2 \ge \dots \ge \lambda_k \ge 1$. The minimum number of vertices in $P_n$ required to host these components and their mandatory gaps is $\sum_{i=1}^k \lambda_i + k - 1 \le n$. Rearranging this inequality yields $\sum_{i=1}^k (\lambda_i + 1) \le n+1$.

This natural inequality establishes a canonical bijection between the isomorphism classes of induced subgraphs of $P_n$ and the integer partitions of $n+1$. Let $p(m)$ denote the partition number of $m$. We map the subgraph sequence $\lambda$ to a partition $\mu \vdash n+1$ by taking the parts $\mu_i = \lambda_i + 1$ for $1 \le i \le k$, and padding the remainder of the integer sum $n+1 - \sum_{i=1}^k (\lambda_i + 1)$ with parts of size $1$. Conversely, removing all $1$s from a partition $\mu \vdash n+1$ and subtracting $1$ from each remaining part perfectly recovers the induced subgraph $\lambda$. Thus, the number of regular $\mathcal{J}$-classes is exactly $p(n+1)$.

By Theorem~5.2, the partial order of principal ideals dictates $J_g \le J_f$ if and only if the induced image subgraph of $g$ is an induced subgraph of the induced image subgraph of $f$.



Thus every component corresponding to $g$ is a subpath of some component of $f$. This corresponds to splitting a part of $\mu \vdash n+1$ into smaller integer summands. Thus, the regular $\mathcal{J}$-class poset of $MR(P_n)$ is order-isomorphic to the integer partition lattice of $n+1$ under refinement.
\end{proof}

\subsection{The Case of $MR(P_6)$}

To have a concrete example, we look at the path graph $P_6$.  The regular $\mathcal{J}$-classes correspond to the $p(6+1) = p(7) = 15$ integer partitions of $7$. Table~\ref{tab:p6_correspondence} lists the 15 realizable regular $\mathcal{J}$-classes, mapping their subgraph partition $\lambda$, their corresponding bijection partition $\mu \vdash 7$, their induced subgraph topology, and their maximal subgroups. 

\begin{table}[ht!]
\centering
\caption{Correspondence between regular $\mathcal{J}$-classes of $MR(P_6)$, induced subgraphs $\lambda$, partitions $\mu \vdash 7$, and maximal subgroups.}
\label{tab:p6_correspondence}
\begin{tabular}{cccll}
\toprule
\textbf{Rank} & \textbf{Subgraph $\lambda$} & \textbf{Partition $\mu \vdash 7$} & \textbf{Induced Subgraph Structure} & \textbf{Maximal Subgroup $H_e$} \\ 
\midrule
6 & $(6)$       & $(7)$         & $P_6$                             & $\mathbb{Z}_2$ \\ \midrule
5 & $(5)$       & $(6,1)$       & $P_5$                             & $\mathbb{Z}_2$ \\
5 & $(4,1)$     & $(5,2)$       & $P_4 \cup P_1$                    & $\mathbb{Z}_2 \times \{1\} \cong \mathbb{Z}_2$ \\
5 & $(3,2)$     & $(4,3)$       & $P_3 \cup P_2$                    & $\mathbb{Z}_2 \times \mathbb{Z}_2 \cong \mathbb{Z}_2^2$ \\ \midrule
4 & $(4)$       & $(5,1,1)$     & $P_4$                             & $\mathbb{Z}_2$ \\
4 & $(3,1)$     & $(4,2,1)$     & $P_3 \cup P_1$                    & $\mathbb{Z}_2$ \\
4 & $(2,2)$     & $(3,3,1)$     & $P_2 \cup P_2$                    & $\mathbb{Z}_2 \wr S_2 \cong D_4$ \\
4 & $(2,1,1)$   & $(3,2,2)$     & $P_2 \cup 2P_1$                   & $\mathbb{Z}_2 \times S_2 \cong \mathbb{Z}_2^2$ \\ \midrule
3 & $(3)$       & $(4,1^3)$     & $P_3$                             & $\mathbb{Z}_2$ \\
3 & $(2,1)$     & $(3,2,1^2)$   & $P_2 \cup P_1$                    & $\mathbb{Z}_2$ \\
3 & $(1,1,1)$   & $(2^3,1)$     & $3P_1$ (Three isolated vertices)  & $S_3$ \\ \midrule
2 & $(2)$       & $(3,1^4)$     & $P_2$                             & $\mathbb{Z}_2$ \\
2 & $(1,1)$     & $(2^2,1^3)$   & $2P_1$ (Two isolated vertices)    & $S_2 \cong \mathbb{Z}_2$ \\ \midrule
1 & $(1)$       & $(2,1^5)$     & $P_1$ (Single isolated vertex)    & $\{1\}$ \\ \midrule
0 & $\emptyset$ & $(1^7)$       & $\emptyset$ (The empty graph)     & $\{1\}$ \\ 
\bottomrule
\end{tabular}
\end{table}

The hierarchical ideal dominance relation governing these regular $\mathcal{J}$-classes forms the exact integer partition refinement lattice for $n=7$, illustrated in Figure~\ref{fig:hasse_p6}.

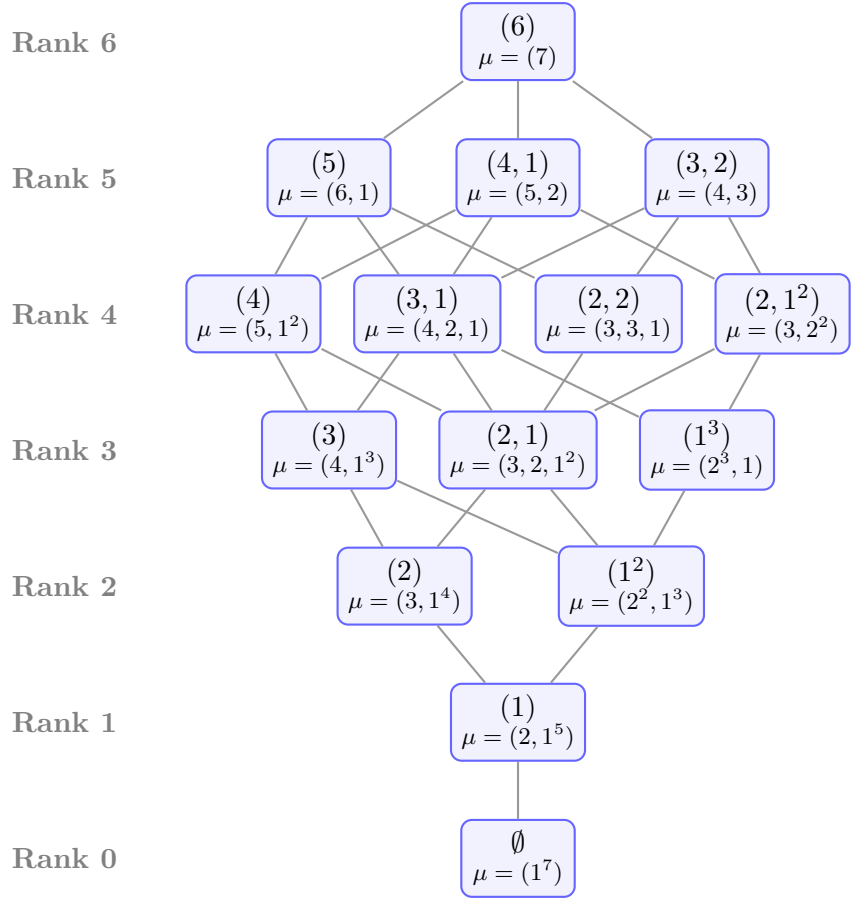
\begin{figure}[ht!]
\centering
\begin{tikzpicture}[
    node distance=1.5cm and 1.2cm,
    vertex/.style={rectangle, draw=blue!60, fill=blue!5, thick, rounded corners, minimum width=15mm, minimum height=8mm, align=center, font=\small},
    edge/.style={-, thick, draw=gray!80}
]

    \node[vertex] (p6) at (0, 10.0) {$(6)$ \\[-0.5ex] {\scriptsize $\mu=(7)$}};

    \node[vertex] (p5)   at (-2.5, 8.2) {$(5)$ \\[-0.5ex] {\scriptsize $\mu=(6,1)$}};
    \node[vertex] (p41)  at (0, 8.2)  {$(4,1)$ \\[-0.5ex] {\scriptsize $\mu=(5,2)$}};
    \node[vertex] (p32)  at (2.5, 8.2)  {$(3,2)$ \\[-0.5ex] {\scriptsize $\mu=(4,3)$}};

    \node[vertex] (p4)   at (-3.5, 6.4) {$(4)$ \\[-0.5ex] {\scriptsize $\mu=(5,1^2)$}};
    \node[vertex] (p31)  at (-1.2, 6.4) {$(3,1)$ \\[-0.5ex] {\scriptsize $\mu=(4,2,1)$}};
    \node[vertex] (p22)  at (1.2, 6.4)  {$(2,2)$ \\[-0.5ex] {\scriptsize $\mu=(3,3,1)$}};
    \node[vertex] (p211) at (3.5, 6.4)  {$(2,1^2)$ \\[-0.5ex] {\scriptsize $\mu=(3,2^2)$}};

    \node[vertex] (p3)   at (-2.5, 4.6) {$(3)$ \\[-0.5ex] {\scriptsize $\mu=(4,1^3)$}};
    \node[vertex] (p21)  at (0, 4.6)  {$(2,1)$ \\[-0.5ex] {\scriptsize $\mu=(3,2,1^2)$}};
    \node[vertex] (p111) at (2.5, 4.6)  {$(1^3)$ \\[-0.5ex] {\scriptsize $\mu=(2^3,1)$}};

    \node[vertex] (p2)   at (-1.5, 2.8) {$(2)$ \\[-0.5ex] {\scriptsize $\mu=(3,1^4)$}};
    \node[vertex] (p11)  at (1.5, 2.8)  {$(1^2)$ \\[-0.5ex] {\scriptsize $\mu=(2^2,1^3)$}};

    \node[vertex] (p1)   at (0, 1.0) {$(1)$ \\[-0.5ex] {\scriptsize $\mu=(2,1^5)$}};

    \node[vertex] (p0)   at (0, -0.8) {$\emptyset$ \\[-0.5ex] {\scriptsize $\mu=(1^7)$}};

    \node[font=\bfseries\small, gray] at (-6, 10.0) {Rank 6};
    \node[font=\bfseries\small, gray] at (-6, 8.2)  {Rank 5};
    \node[font=\bfseries\small, gray] at (-6, 6.4)  {Rank 4};
    \node[font=\bfseries\small, gray] at (-6, 4.6)  {Rank 3};
    \node[font=\bfseries\small, gray] at (-6, 2.8)  {Rank 2};
    \node[font=\bfseries\small, gray] at (-6, 1.0)  {Rank 1};
    \node[font=\bfseries\small, gray] at (-6, -0.8) {Rank 0};

    \draw[edge] (p6) -- (p5);
    \draw[edge] (p6) -- (p41);
    \draw[edge] (p6) -- (p32);
    
    \draw[edge] (p5) -- (p4);
    \draw[edge] (p5) -- (p31);
    \draw[edge] (p5) -- (p22);
    
    \draw[edge] (p41) -- (p4);
    \draw[edge] (p41) -- (p31);
    \draw[edge] (p41) -- (p211);
    
    \draw[edge] (p32) -- (p31);
    \draw[edge] (p32) -- (p22);
    \draw[edge] (p32) -- (p211);

    \draw[edge] (p4) -- (p3);
    \draw[edge] (p4) -- (p21);
    
    \draw[edge] (p31) -- (p3);
    \draw[edge] (p31) -- (p21);
    \draw[edge] (p31) -- (p111);
    
    \draw[edge] (p22) -- (p21);
    
    \draw[edge] (p211) -- (p21);
    \draw[edge] (p211) -- (p111);

    \draw[edge] (p3) -- (p2);
    \draw[edge] (p3) -- (p11);
    
    \draw[edge] (p21) -- (p2);
    \draw[edge] (p21) -- (p11);
    
    \draw[edge] (p111) -- (p11);

    \draw[edge] (p2) -- (p1);
    \draw[edge] (p11) -- (p1);

    \draw[edge] (p1) -- (p0);

\end{tikzpicture}
\caption{The Hasse diagram of the regular $\mathcal{J}$-class poset of $MR(P_6)$, annotated with the isomorphic integer partitions $\mu \vdash 7$.}
\label{fig:hasse_p6}
\end{figure}

\section{Characterization of Regularity in \texorpdfstring{$St(\Gamma)$}{St(Gamma)}}
\label{sec:regularity_thorough2}

In this section, we study the regularity of the strict continuous semigroup $St(\Gamma)$ for a finite simple graph $\Gamma = (V, E)$. In contrast to the case of $\operatorname{MR}(\Gamma)$ in Lemma \ref{RegJMR}, the  regular $\mathcal{J}$-classes of $\StG{\Gamma}$ form a chain.


\begin{theorem}\label{strictreg}
Let $f \in \StG\Gamma$ be an element with domain matching $M = \{E_1, \dots, E_k\}$ and image set $I = \{v_1, \dots, v_k\}$ with $xf=v_{i}$ if $x \in E_i$. Then $f$ is a regular element of $St(\Gamma)$ if and only if the following two conditions hold:
\begin{enumerate}
    \item \textbf{The Transversal Condition:} The domain matching $M$ admits an independent transversal; that is, there exists an independent set of vertices $I' = \{u_1, \dots, u_k\} \subseteq V$ such that $u_i \in E_i$ for all $i \in \{1, \dots, k\}$.
    \item \textbf{The Covering Condition:} The image set $I$ admits a covering matching within $\Gamma$; that is, there exists a matching $M' = \{E'_1, \dots, E'_k\}$ in $\Gamma$ of size $k$ such that $v_i \in E'_i$ for all $i \in \{1, \dots, k\}$.
\end{enumerate}
\end{theorem}

\begin{proof}
 Suppose $f$ is regular and let $g \in St(\Gamma)$ satisfy $fgf = f$. Let $\text{Dom}(g) = M_g$ and $\text{Im}(g) = I_g$. For any $x \in \bigcup M$, we have $xf = v_i \in I$ if $x \in E_i$. To preserve the rank under $xfgf = v_i$, the vertex $v_i$ must lie in the domain of $g$. Since $g \in St(\Gamma)$, $\text{Dom}(g)$ is a matching, meaning each $v_i$ must belong to some edge $E'_i \in M_g$. Because $I$ has size $k$ and no edge can contain more than one vertex of an independent set, there must exist a sub-matching $M' \subseteq M_g$ of size exactly $k$ such that $v_i \in E'_i$, which establishes the \textit{Covering Condition}. 

Furthermore, let $u_i = (v_i)g$. Because $g \in St(\Gamma)$, the set $I' = \{u_1, \dots, u_k\}$ is a subset of the independent set $I_g$, and hence $I'$ is independent. Substituting this into the regularity equation gives $(x)fgf = (v_i)gf = (u_i)f = v_i$, which implies $u_i \in E_i$. Thus, $I'$ is an independent transversal for $M$, establishing the \textit{Transversal Condition}.

Conversely, suppose both conditions hold. Let $I' = \{u_1, \dots, u_k\}$ be the independent transversal of $M$ ($u_i \in E_i$) and let $M' = \{E'_1, \dots, E'_k\}$ be the covering matching of $I$ ($v_i \in E'_i$). Define a partial map $g: V \to V$ by:
\[
yg = \begin{cases} 
u_i & \text{if } y \in E'_i \text{ for some } i \in \{1, \dots, k\} \\ 
undefined & \text{otherwise} 
\end{cases}
\]
Since $\text{Dom}(g) = M'$ is a valid matching and $\text{Im}(g) = I'$ is an independent set, $g \in St(\Gamma)$. For any $x \in \bigcup M$, let $x \in E_i$. Then $xfgf = v_ig = u_if = v_i = xf$. For any $x \notin \bigcup M$, $xfgf$ and $xf$ are undefined. Thus $fgf = f$, and $f$ is regular.
\end{proof}

\begin{theorem}
The regular $\mathcal{J}$-classes of $\StG{\Gamma}$ form a chain:
\[
J_0 < J_1 < J_2 < \dots < J_k
\]
where $k$ is the largest rank of an element of $St(\Gamma)$ satisfying the Transversal Condition and the Covering Condition. Two regular elements $f, g \in \StG{\Gamma}$ belong to the same $\mathcal{J}$-class if and only if their underlying matchings have the same size, i.e., $|M_f| = |M_g|$.
\end{theorem}

\begin{proof}


Let $e_{k}$ be an idempotent of rank $k$ in $\StG\Gamma$. Then there is a matching $M=
\{E_{1}\ldots, E_{k}\}$ and an independent set $\{v_{1},\ldots, v_{k}\}$ such that $e_{k}$ sends each vertex of $E_i$ to $v_i \in E_{i}$. The restriction of $e_{k}$ to $E_{1} \cup \ldots E_{r}$ is an idempotent $e_r$ of rank $r$ in $\StG\Gamma$. Let $J_r$ be the $\mathcal{J}$-class of $e_{r}$. Then ${J}_{0} < \ldots \mathcal <{J}_{k}$.

We claim that if $f$ is a regular element of $St(\Gamma)$, then $f\mathcal{J}e_{r}$ where $r=\operatorname{rank(f)}$. Without loss of generality we can assume that $f$ is an idempotent. $f$ maps a matching $M'=\{\{(E'_{i}=(x_{i},y_{i})\}\mid i=1, \ldots , r\}$ with $x_{i}f=y_{i}f=x_{i}$ where $x_{1}, \ldots , x_{r}$ is an independent set.

Let $g:V\rightarrow V$ be defined by sending both vertices of $E'_{i}$ to $v_{i}$ belongs to $\StG\Gamma$.
The function $h$ that sends each edge $E_{i}$ to $x_{i}$ also belongs to $\StG\Gamma$. We have $f=ge_{r}h$ and $e_{r}=hfg$. Therefore $f\mathcal{J}e_{r}$.

\end{proof}

\subsection{Regularity of the Strict Semigroup of Paths and Cycles}

We prove that for path graphs and cycle graphs their strict continuous semigroups are regular semigroups.

\begin{theorem}
The semigroup $St(P_n)$ is regular for all path graphs $P_n$ ($n \ge 1$).
\end{theorem}

\begin{proof}
Let $f \in St(P_n)$ have rank $k$, domain matching $M = \{E_1, \dots, E_k\}$, and image independent set $I = \{v_1, \dots, v_k\}$. Labeling $V(P_n) = \{1, \dots, n\}$ linearly:
\begin{enumerate}
    \item \textbf{Transversal Condition:} Write each edge as $E_i = \{a_i, a_i+1\}$. Sorting them such that $a_1 < a_2 < \dots < a_k$, the matching condition forces $a_{i+1} - a_i \ge 2$. Selecting the left endpoints $I' = \{a_1, \dots, a_k\}$ guarantees that no two chosen vertices are adjacent. Hence, $I'$ is an independent transversal.
    \item \textbf{Covering Condition:} Order $I$ such that $v_1 < v_2 < \dots < v_k$. We execute a greedy left-to-right matching construction: for $j=1$ to $k$, we pair $v_j$ with $v_j - 1$ if unassigned; otherwise, we pair it with $v_j + 1$. This matching sequence can only fail if it cascades to the right boundary, forcing $v_k = n$ to look for $n+1$, which requires $I = \{1, 3, 5, \dots, n\}$ where $n$ is odd. This forces $|I| = k = \frac{n+1}{2}$. However, the maximum matching size of $P_n$ is $\lfloor \frac{n}{2} \rfloor = \frac{n-1}{2}$, creating the contradiction $k > \lfloor \frac{n}{2} \rfloor$. Thus, a covering matching $M'$ always exists.
\end{enumerate}
Since both conditions are unconditionally satisfied, $St(P_n)$ is regular.
\end{proof}

\begin{theorem}
The semigroup \texorpdfstring{$St(C_n)$}{St(Cn)} is regular for all cycle graphs \texorpdfstring{$C_n$}{Cn} (\texorpdfstring{$n \ge 3$}{n >= 3}).
\end{theorem}

\begin{proof}
Let $f \in St(C_n)$ have rank $k$, domain matching $M$, and independent image set $I$. Fix a clockwise orientation on $C_n$.
\begin{enumerate}
    \item \textbf{Transversal Condition:} For each edge $E_i \in M$, let $u_i \in E_i$ be the first vertex encountered under the clockwise orientation. Because the edges are vertex-disjoint, the cyclic path distance between any two distinct elements in $I' = \{u_1, \dots, u_k\}$ is at least $2$. Thus, $I'$ is always an independent transversal.
    \item \textbf{Covering Condition:} Since a domain matching $M$ of size $k$ exists, we have $k \le \lfloor \frac{n}{2} \rfloor$. We construct a covering matching $M'$ based on the size of $I$:
    \begin{itemize}
        \item \textbf{Case 1 ($k < \frac{n}{2}$):} There is at least one vertex $z \notin I$. Deleting $z$ embeds $I$ into a path subgraph $P_{n-1}$. Because $k \le \lfloor \frac{n-1}{2} \rfloor$, the path greedy matching algorithm is guaranteed to safely build $M'$.
        \item \textbf{Case 2 ($n$ is even and $k = \frac{n}{2}$):} Independence forces $I$ to consist of exactly all alternating vertices. Pairing each $v_i \in I$ with its clockwise neighbor yields a perfect matching $M'$ covering $I$.
        \item \textbf{Case 3 ($n$ is odd and $k = \frac{n-1}{2}$):} The maximal independent set leaves exactly one adjacent pair of empty vertices $z_1, z_2 \notin I$. Slicing the cycle along edge $\{z_1, z_2\}$ yields a path $P_n$ containing $I$. The path greedy matching algorithm initiated at the vertex adjacent to $z_1$ terminates perfectly at $z_2$ without boundary failure, producing $M'$.
    \end{itemize}
\end{enumerate}
Thus, $St(C_n)$ satisfies both conditions and is regular.
\end{proof}

\subsection{A Counterexample to Regularity}

To demonstrate that regularity is not universal across all graph families, we present a counterexample.

\begin{proposition}\label{not reg}
There exist finite simple graphs $\Gamma$ for which \texorpdfstring{$St(\Gamma)$}{St(Gamma)} is not a regular semigroup.
\end{proposition}

\begin{proof}
Let $\Gamma = P_{3} \cup P_{2}$ where the vertices of $P_{3}$ are $\{0,1,2\}$ and the vertices of $P_{2}$ are  $\{x,y\}$. Define an element $f \in \StG\Gamma$ where its matching is $M = \{\{0,1\}, \{x, y\}\}$ and independent set $I = \{0,2\}$. More precisely, let $(\{0,1\})f = \{0\}$ and $(\{x, y\})f = \{2\}$. This map is of rank $2$, and its non-empty preimages are edges, so $f \in \StG\Gamma$.

For $f$ to be regular, its image $I = \{0, 2\}$ must satisfy the Covering Condition by admitting a matching $M'$ of size 2. However the unique edge covering $0$ must contain $1$ and the unique edge covering $2$ must also contain $1$. Consequently, $I$ cannot be covered by a matching of size 2. The Covering Condition fails and $f$ is not a regular element. Thus, $\StG\Gamma$ is not regular.
\end{proof}

\subsection*{Open Problem: Classification of Regularity}
The previous examples show that regularity of $\StG\Gamma$ is related to restrictions on specific induced subgraphs. We conclude this section by posing the comprehensive global classification as an open problem:

\begin{openproblem}
 Characterize the complete class of finite simple graphs $\Gamma$ such that the strict continuous partial semigroup $St(\Gamma)$ is a regular semigroup.
\end{openproblem}

We answer this open problem when the graph is bipartite. Let $G = (A \cup B, E)$ be a bipartite graph. Let $M = \{e_1, e_2, \dots, e_k\}$ be a matching in $G$ of size $k$. 

Let $I = \{v_1, \dots, v_k\}$ be an arbitrary transversal of $M$ (i.e., taking exactly one endpoint from each edge of $M$), which we will consider as our image set. Note that $I$ may contain vertices from both $A$ and $B$. We define the subsets $I_A = I \cap A$ and $I_B = I \cap B$ so that $I=I_{A} \cup I_{B}$. The transversal condition is obvious for bipartite graphs.

\begin{theorem}
Every matching $M$ in a bipartite graph $\Gamma$ admits an independent transversal.
\end{theorem}

\begin{proof}
Because $G$ is bipartite both $I_A$ and $I_B$ are transversals of $M$.

\end{proof}

We turn to the covering condition. This turns out to be equivalent to a modified version of Hall's Marriage Theorem.

\begin{theorem}\label{bipartcover}
An arbitrary image set $I = I_A \cup I_B$ admits a covering matching $M'$ of size $k$ in $G$ if and only if Hall's condition holds independently for both subsets mapping into the remaining graph; namely:
\begin{align*}
\forall S_A \subseteq I_A, \quad |N_G(S_A) \setminus I_B| &\ge |S_A| \\
\forall S_B \subseteq I_B, \quad |N_G(S_B) \setminus I_A| &\ge |S_B|
\end{align*}
where $N_G(S)$ denotes the neighborhood of $S$ in $G$.
\end{theorem}

\begin{proof}
We require a matching $M'$ in $G$ of size $k$ such that each vertex $v_i \in I$ belongs to a distinct edge in $M'$. Clearly vertices in $I_A$ must be matched exclusively to vertices in $B \setminus I_B$, and vertices in $I_B$ must be matched exclusively to vertices in $A \setminus I_A$.

\vspace{1em}

Assume $M'$ is such a covering matching of size $k$. For any subset $S_A \subseteq I_A$, $M'$ assigns each vertex in $S_A$ a unique neighbor in $B$. Since no edge in $M'$ can connect to a vertex in $I_B$, these matched neighbors lie entirely in $B \setminus I_B$. Thus, the set of these matched neighbors has size $|S_A|$ and is a subset of $N_G(S_A) \setminus I_B$, implying $|N_G(S_A) \setminus I_B| \ge |S_A|$. By symmetry, the same holds for any subset $S_B \subseteq I_B$.

Conversely assume the modified Hall's conditions hold. First, consider the bipartite subgraph induced by the partitions $I_A$ and $B \setminus I_B$. The condition $\forall S_A \subseteq I_A, |N_G(S_A) \setminus I_B| \ge |S_A|$ allows us to apply Hall's Marriage Theorem to find a matching $M'_A$ that saturates $I_A$.

Second, consider the bipartite subgraph induced by the partitions $I_B$ and $A \setminus I_A$. The condition $\forall S_B \subseteq I_B, |N_G(S_B) \setminus I_A| \ge |S_B|$ allows us to find a matching $M'_B$ that saturates $I_B$.

Because $M'_A$ uses vertices only in $I_A \cup (B \setminus I_B)$, and $M'_B$ uses vertices only in $I_B \cup (A \setminus I_A)$, the two matchings are strictly vertex-disjoint. Their union $M' = M'_A \cup M'_B$ is a valid matching in $G$ that covers all vertices of $I$. Because exactly one endpoint of each edge in $M'$ lies in $I$, $M'$ contains exactly $|I_A| + |I_B| = k$ edges.
\end{proof}

\section{The Order of the Strict Continuous Semigroup \texorpdfstring{$St(P_n)$}{St(Pn)}  and \texorpdfstring{$St(C_n)$}{St(Cn)}}
\label{sec:strict_Pn}

In this section, we compute the exact order of $St(P_n)$, the semigroup of all strict continuous partial maps on a path graph. We adopt the definition where the path $P_n$ consists of $n$ vertices $V = \{1, 2, \dots, n\}$ and $n-1$ edges $E = \{\{i, i+1\} \mid 1 \le i \le n-1\}$.

To determine the order of $St(P_n)$, we partition the semigroup into its  $\mathcal{J}$-classes for all ranks $k$. By the Rees matrix theorem, the cardinality of any principal $\mathcal{J}$-class of rank $k$ is given exactly by the product of the number of its $\mathcal{R}$-classes, its $\mathcal{L}$-classes, and the order of its maximal structure group:
\begin{equation}
|\mathcal{J}_k| = |\mathcal{R}_k| \times |\mathcal{L}_k| \times |G_k|
\end{equation}

For $St(P_n)$, these components possess precise combinatorial interpretations:

\begin{itemize}
    \item \textbf{Number of $\mathcal{R}$-classes ($|\mathcal{R}_k|$):} Each $\mathcal{R}$-class corresponds uniquely to the domain of the maps in the class, which must be a matching of size $k$. This requires choosing $k$ pairwise non-adjacent edges out of the $n-1$ total sequential edges in the path. The number of ways to select $k$ non-consecutive items from a linear sequence of length $n-1$ is given by the binomial coefficient:
    \[
    \binom{(n-1)-k+1}{k} = \binom{n-k}{k}
    \]
    
    \item \textbf{Number of $\mathcal{L}$-classes ($|\mathcal{L}_k|$):} Each $\mathcal{L}$-class corresponds uniquely to the image of the maps in the class. We claim that the image $\operatorname{Im}(f)$ must form an independent set in the target graph $P_n$. Suppose, for the sake of contradiction, that $\operatorname{Im}(f)$ contains two adjacent vertices $u$ and $v$. The target edge $e = \{u, v\}$ must have a valid preimage under $f$. By definition, $ef^{-1}$ contains $uf^{-1} \cup vf^{-1}$. Since both $uf^{-1}$ and $vf^{-1}$ are strictly required to be disjoint edges in the domain, their union forms a disconnected space of at least four vertices. This union cannot form a valid single vertex or single edge, thereby violating the global continuity condition of $MR(P_n)$. Thus, $\operatorname{Im}(f)$ contains no adjacent vertices. Selecting $k$ non-adjacent vertices out of the $n$ sequential vertices yields distinct choices given by:
    \[
    \binom{n-k+1}{k}
    \]
    
    \item \textbf{Maximal Subgroup ($|G_k|$):} The structure group for a rank-$k$ partial map is the symmetric group $G_k \cong S_k$. This represents the internal permutations and bijective alignments between the chosen domain matching and the target independent set. Its order is exactly $k!$.
\end{itemize}

Summing across all valid ranks $k$ up to the maximum possible independent matching size $\lfloor n/2 \rfloor$, the total order of the strict matching semigroup is given by the closed-form formula:

\begin{equation}
|St(P_n)| = \sum_{k=0}^{\lfloor n/2 \rfloor} \binom{n-k}{k} \binom{n-k+1}{k} k!
\end{equation}

Evaluating this formula provides the structural sequence of orders for strict continuous maps on small path graphs. For the rank $k=0$ term, the product always evaluates to $1$, corresponding to the unique empty map $\emptyset$. For $n=1$ to $6$ we have the following.

\begin{itemize}
    \item \textbf{For $P_1$} ($1$ vertex, $0$ edges): 
    $|St(P_1)| = \binom{1}{0}\binom{2}{0}0! = 1$
    
    \item \textbf{For $P_2$} ($2$ vertices, $1$ edge): 
    $|St(P_2)| = \binom{2}{0}\binom{3}{0}0! + \binom{1}{1}\binom{2}{1}1! = 1 + 2 = 3$
    
    \item \textbf{For $P_3$} ($3$ vertices, $2$ edges): 
    $|St(P_3)| = \binom{3}{0}\binom{4}{0}0! + \binom{2}{1}\binom{3}{1}1! = 1 + 6 = 7$
    
    \item \textbf{For $P_4$} ($4$ vertices, $3$ edges): 
    $|St(P_4)| = 1 + \binom{3}{1}\binom{4}{1}1! + \binom{2}{2}\binom{3}{2}2! = 1 + 12 + 6 = 19$
    
    \item \textbf{For $P_5$} ($5$ vertices, $4$ edges): 
    $|St(P_5)| = 1 + \binom{4}{1}\binom{5}{1}1! + \binom{3}{2}\binom{4}{2}2! = 1 + 20 + 36 = 57$
    
    \item \textbf{For $P_6$} ($6$ vertices, $5$ edges): 
    $|St(P_6)| = 1 + \binom{5}{1}\binom{6}{1}1! + \binom{4}{2}\binom{5}{2}2! + \binom{3}{3}\binom{4}{3}3! = 1 + 30 + 120 + 24 = 175$
\end{itemize}

A similar analysis can be done for the cardinality of $St(C_{n})$.
One obtains the following results. For $n=1,2$, $C_{n}=P_{n}$. For $n\ge 3$ we have:

\begin{equation*}
                    |St(C_n)| = \sum_{k=0}^{\lfloor n/2 \rfloor} \left( \frac{n}{n-k} \binom{n-k}{k} \right)^2 k!
                \end{equation*}

The sequence begins 1,3,10,25,76,223.

\subsection{Connections to Graph Theory and Related Areas}

The classical concepts of matchings and independent sets intersect profoundly when examining the structural pairs that define the elements of the strict matching semigroup $\operatorname{St}({\Gamma})$. Despite its natural definition, as far as we know, this semigroup has never been considered before. We look at connections with some mathematical problems in graph theory and related fields. We remark the the semigroup $\operatorname{St}(\Gamma)$ generalizes in an obvious way for any hypergraph.

\subsubsection{Strong Matchings and Cross-Cycles}
In graph theory, a matching $M$ that allows one to select exactly one endpoint from each edge to form an independent set $I$ is structurally linked to an \textit{Induced Matching} (or Strong Matching). In the specific context of independence complexes, an induced matching that forms a cross-section over a maximal independent set—where precisely one vertex from each matching edge belongs to $I$—is defined as a \textit{cross-cycle}. 

This algebraic requirement imposes a rigid topological constraint on the graph layout: the only edges allowed to exist between the vertices saturated by the matching $M$ are the matching edges themselves, or edges connecting the unchosen endpoints. The subgraph induced exclusively by the chosen independent representatives $I$ must contain absolutely zero edges, ensuring that the partial transformation does not map adjacent domain points into conflicting states under the right-action convention.

\subsubsection{Structural Interpretations and Information Theory}
\begin{itemize}
    \item \textbf{Independent Transversals:} In hypergraph theory and fractional combinatorics, asking for a choice of exactly one vertex per matching edge to form an independent set is structurally framed as finding an \textit{Independent Transversal}. A transversal is a vertex set containing at least one element from a collection of pre-defined subsets (here, the individual edge components of $M$). The added condition that this transversal must be an independent set binds the global structure tightly to the internal topology of $G$.
    \item \textbf{Information Theory and Shannon Capacity:} In information theory and channel coding, this configuration models conflict-free data transmission over a noisy channel. If the edges in the matching $M$ represent pairs of primary symbols that are easily confused during transmission, picking exactly one vertex from each edge represents isolating a safe, unique vocabulary. The independent set requirement ensures that no two selected codewords interfere via secondary constraints, maximizing the error-free throughput.
    \item \textbf{Bounding the Independence Number:} If a graph $G$ admits a valid matching-independent set pair $(M, I)$ of rank $k$ within the semigroup $\StG{G}$, it provides an immediate structural lower bound on the global independence number of the graph, ensuring that $\alpha(G) \ge k$.
\end{itemize}

\subsection{Computational Efficiency Across Restricted Graph Classes}
Because finding a maximum matching with independent representatives is structurally bounded by the Maximum Independent Set problem, determining the largest valid rank $k$ for an element $f_{M,I} \in \StG{G}$ is broadly NP-hard on general graphs. Finding a maximum induced matching remains NP-hard even when restricted to highly structured families such as chordal graphs, co-chordal graphs, and co-bipartite graphs. However, explicit graph subfamilies have been identified where the problem crosses the computational boundary into polynomial-time tractability:

\subsubsection{Trees, Forests, and Interval Graphs}
While general chordal graphs remain computationally intractable, specific linear subclasses admit high efficiency. Maximum Induced Matchings can be solved in polynomial time on \textit{interval graphs} using a specialized stack-based interval scheduling algorithm that bypasses the high computational overhead of explicitly constructing the graph's line graph. Furthermore, for trees and forests, the problem becomes solvable in strict linear time $\mathcal{O}(|V|)$ using a bottom-up dynamic programming approach that evaluates candidate matching-independent set pairs incrementally from the leaves to the root.

\subsubsection{Forbidden Subgraphs and Bounded-Degree Planar Graphs}
If the problem is constrained such that every edge not in the matching $M$ must share exactly one vertex with an edge in $M$, it generalizes to the \textit{Dominating Induced Matching} problem. This specific structural variation can be solved in polynomial time for graphs that exclude a `long claw' (an induced $K_{1,3}$ with elongated paths) as an induced subgraph. This is achieved via specialized graph reductions that map the independent choices directly to saturating matchings in auxiliary bipartite graphs.

Additionally, bounding the maximum degree of planar graphs guarantees efficiency when exact maximums are not strictly required. In \textit{subcubic planar graphs} (where the maximum vertex degree is $3$), a guaranteed lower-bound induced matching can be computed in polynomial time by partitioning the global edge set into independent matchings using strong edge coloring algorithms.

\section{Basics of Krohn-Rhodes Complexity}

The Krohn-Rhodes decomposition theorem is a cornerstone of the algebraic theory of finite semigroups, providing a hierarchical decomposition analogous to the Jordan-H\"older theorem for finite groups. While extensive expositions on the foundations of this theory are thoroughly documented in the literature (see, e.g., \cite{qtheory, Eilenberg}), we briefly recall the core definitions and conventions necessary for our development.
Readers familiar with Krohn-Rhodes theory may skip this and the next section.

Let $S$ and $T$ be finite semigroups. We say that $S$ \emph{divides} $T$, denoted $S \prec T$, if $S$ is a homomorphic image of a subsemigroup of $T$. 

The primary structural tool used in the decomposition is the \emph{wreath product}. Let $(X, S)$ and $(Y, T)$ be transformation semigroups acting on the right of the sets $X$ and $Y$, respectively. The wreath product $(X, S) \wr (Y, T)$ is defined as the transformation semigroup $(X \times Y, W)$, where the underlying semigroup is $W = S^Y \times T$, acting on $X \times Y$ via the rule:
\[ (x, y)(f, t) = (x(yf), yt) \]
for all $x \in X, y \in Y, f: Y \to S$, and $t \in T$. For abstract semigroups $S$ and $T$, their abstract wreath product $S \wr T$ is defined by considering their natural actions on their underlying sets adjoined with an identity element where necessary.

A finite semigroup $A$ is called \emph{aperiodic} (or combinatorial) if it contains no non-trivial subgroups. Equivalently, $A$ is aperiodic if there exists a sufficiently large integer $n \ge 1$ such that $a^n = a^{n+1}$ for all $a \in A$. 

The fundamental theorem of Krohn and Rhodes states that every finite semigroup $S$ divides an alternating wreath product chain of finite groups and finite aperiodic semigroups:
\[ S \prec A_n \wr G_n \wr A_{n-1} \wr \dots \wr G_1 \wr A_0 \]
where each $G_i$ is a finite group and each $A_i$ is a finite aperiodic semigroup. 

The \emph{Krohn-Rhodes complexity} of $S$, denoted $Sc$, is defined as the minimum number of group blocks $G_i$ required across all valid alternating wreath product decompositions that $S$ divides. By definition, if $S$ is aperiodic, its group-complexity is zero ($Sc = 0$). The complexity operator $c$ is well-behaved with respect to standard algebraic constructions, satisfying:
\begin{enumerate}
    \item If $S \prec T$, then $Sc \le Tc$.
    \item $(S \times T)c = \max(Sc, Tc)$.
    \item $(S \wr T)c \le Sc + Tc$.
\end{enumerate}

\subsection{ Complexity of Transformation Semigroups of Degree at Most 2}

We recall the following definition.

\begin{definition}
Let $X = (Q, S)$ be a transformation semigroup. For any partial function $s \in S$, the \emph{kernel} of $s$, denoted $\ker(s)$, is the partition on the domain of $s$ given by $q \mathrel{\ker(s)} q'$ if and only if $qs = q's$. The \emph{degree} of $s$, denoted $sd$, is the maximal cardinality of a partition class of $\ker(s)$. The degree of the transformation semigroup, denoted $Xd$, is defined as:
\[ Xd = \max_{s \in S} sd \]
\end{definition}

\begin{theorem}[\cite{TilsonXII}]
Let $X = (Q, S)$ be a transformation semigroup. Then:
\[ Sc \le Xd \]
\end{theorem}

We achieve a  bound on complexity for the semigroups that we are considering in this paper.

\begin{theorem}\label{thesis}

Let $\Gamma$ be a graph. Then we have the following.

\begin{enumerate}

\item{$\operatorname{MR}(\Gamma)c \le 2$.}

\item{$\StG{\Gamma}c \leq 1$.}

\item{$\operatorname{Inj}(\Gamma)c \le 1$}
    
\end{enumerate}
    
\end{theorem}

\begin{proof}

Since $\operatorname{MR}(\Gamma)$ is a semigroup of degree at most 2, (1) follows from Theorem  \ref{thesis} for the case that $d \leq 2$. (2) is Corollary 3.7 of \cite{cremona}. (3) follows since $\operatorname{Inj}(\Gamma)$ is a submonoid of $\operatorname{Sim}(V(\Gamma))$ and it is known \cite{qtheory} that every inverse monoid has complexity at most 1.
    
\end{proof}

The main purpose of the sections on complexity theory is to determine conditions that $\\operatorname{MR}(\Gamma)c \leq 1$. We concentrate on the case of paths and cycles.

\subsection{Lower Bounds to Complexity}

While upper bounds correspond to the existence of valid decompositions, lower bounds establish structural obstructions that force a semigroup to require a minimum number of group components \cite{KR.1968, TilsonXII}. These bounds are formulated using type I and type II subsemigroups \cite{TilsonXII}.

Let $S$ be a finite semigroup.
\begin{enumerate}
    \item A subsemigroup $T$ of $S$ is an \emph{absolute type I subsemigroup} if it is generated by a chain of $\mathcal{L}$-classes of $T$.
    \item The \emph{type II subsemigroup} $S_{II}$ of $S$ is the smallest subsemigroup of $S$ containing all the idempotents of $S$ and closed under weak conjugation: if $xyx = x$ for $x, y \in S$, then $x S_{II} y \cup y S_{II} x \subseteq S_{II}$.
\end{enumerate}

By Ash's Theorem \cite{Ash}, $S_{II}$ is precisely the group kernel of $S$, consisting of all elements related to the identity under any relational morphism from $S$ to a finite group \cite{TilsonXII}. This closure property allows Rhodes and Tilson to define a systematic lower bound, $Sl$ \cite{TilsonXII}.

\begin{theorem}[Rhodes and Tilson 1972 \cite{TilsonXII}]\label{lb2}
Let $S$ be a finite semigroup. Define $Sl$ to be the largest integer $k$ such that there exists a chain of subsemigroups:
\[ S \ge T_1 > (T_1)_{II} > T_2 > (T_2)_{II} > \dots > T_k > (T_k)_{II} \]
where each $T_i$ is a non-aperiodic absolute type I subsemigroup and $(T_k)_{II}$ is not aperiodic. Then:
\[ Sl \le Sc \]
\end{theorem}

In the specific case where $S$ divides the wreath product of an aperiodic semigroup and a group in either order, these lower bounds perfectly characterize these classes.

\begin{theorem}[Rhodes-Tilson 1972, Rhodes-Karnofsky 1979 \cite{TilsonXII}]
Let $S$ be a finite semigroup.
\begin{enumerate}
    \item $S$ divides a wreath product of the form $A \wr G$, where $A$ is aperiodic and $G$ is a group, if and only if $S_{II}$ is aperiodic. In this case, $Sc = Sl \le 1$.
    \item $S$ divides a wreath product of the form $G \wr A$, where $G$ is a group and $A$ is aperiodic, if and only if $S^{\mathcal{L}'}$ is aperiodic, where $S^{\mathcal{L}'}$ is the structural image of $S$ acting partially on its regular $\mathcal{L}$-classes. In this case, $Sc = S\theta \le 1$.
\end{enumerate}
\end{theorem}

\section{Small Monoids, Right-Orbit Monoids, and Tilson's Theorem}

A monoid $M$ that is the union of a group of units $U(M)$ and a unique regular 0-minimal ideal $M^{0}(G,A,B,C)$ is called a {\em small monoid}. They play an important role in combinatorial applications of semigroup theory and in its representation theory. Tilson's Theorem gives a necessary and sufficient condition for a small monoid to have complexity 1. The proof of this theorem was the first that used tools related to the presentation lemma  and flow theory \cite{pl, flows, FlowsI, FlowsII} that are crucial to the rest of the paper. The next definition defines a small monoid associated to every regular $\mathcal{J}$-class of a monoid $M$.

\begin{definition}
Let $M$ be a finite monoid, and let $J$ be a regular $\mathcal{J}$-class of $M$. Let $J^0$ denote the principal factor corresponding to $J$, which carries the structure of a Rees matrix semigroup $M^0(G, A, B, C)$. The \emph{small monoid} corresponding to $J$, denoted by $\operatorname{Sm}(J)$, is defined as the set:
\[ \operatorname{Sm}(J) = U(M) \cup J^0 \]
where $U(M)$ is the group of units of $M$. The operation on $\operatorname{Sm}(J)$ are inherited directly from $M$, with $U(M)$ acting on the elements of $J^0$ via left and right multiplication, and products falling outside of $J$ evaluating to the zero element $0 \in J^0$.
\end{definition}

Let $M$ be a small monoid. The group of units $H$ acts on the right of the set of $\mathcal{L}$-classes $B$ of the $0$-minimal ideal. This action corresponds to the restriction of the right Schützenberger representation defining the right letter mapping image $\operatorname{RLM}(S)$ to the subgroup $H$ \cite{qtheory}.

\begin{definition}
Let $S = H \cup M^0(G, A, B, C)$ be a small monoid. Let $B_1, B_2, \dots, B_k$ be the distinct orbits of the set of $\mathcal{L}$-classes $B$ under the right action of $H$. For each $\mathcal{L}$-class $L \in B$, the subset $LH$ constitutes an orbit of $H$ on $B$. The submonoids of $S$ of the form:
\[ M_i = H \cup (A \times G \times B_i) \cup \{0\} \]
or equivalently, $H \cup LH \cup \{0\}$, are called the \emph{right-orbit monoids} of $S$.
\end{definition}

Each right-orbit monoid gathers a single orbit of $\mathcal{L}$-classes under the group of units. Tilson's fundamental theorem establishes that the complexity of the entire small monoid collapses to $1$ if and only if all of these component blocks have an aperiodic idempotent generated semigroup.

\begin{theorem}[Tilson's Small Monoid Theorem; \cite{2J}]\label{thm:tilson_small_monoids}
Let $S = H \cup M^0(G, A, B, C)$ be a small monoid. The Krohn-Rhodes complexity of $S$ satisfies $Sc = 1$ if and only if each right-orbit monoid $H \cup LH \cup \{0\}$ contains an aperiodic idempotent-generated subsemigroup.
\end{theorem}

 In the subsequent sections, we will employ Tilson's Theorem to evaluate the exact complexity of cycle graphs $C_n$. In the case of path graphs $P_n$, Tilson's Theorem fails to give a complete answer and we must use techniques from flow theory. We first look at the idempotent generated submonoids of Margolis-Rhodes monoids of cycles and paths.

\section{Non-Aperiodicity of the Idempotent Generated  Subsemigroup of  $MR(C_{n})$ and the Complexity of $MR(C_n)$}

\subsection{Preliminaries and Notation}

We will demonstrate that the idempotent-generated subsemigroup $\langle E(MR(C_n)) \rangle$ is not aperiodic for any $n \ge 4$. To do this, we split the construction into even and odd lengths, constructing group elements of order $k = \lfloor n/2 \rfloor$.

\subsection{The Even Case ($n = 2k$)}
Let $n = 2k$ for $k \ge 2$. Because $C_{2k}$ is bipartite, its edge set can be partitioned into two perfect matchings. We can construct an element of order $k$ in the idempotent generated subsemigroup of $MR(C_{n})$ using exactly two idempotents.

\begin{definition}[Idempotents for Even Cycles]
We define two partial idempotents $e_1, e_2 \in St(C_{2k})$ whose domains are exactly the two perfect matchings of $C_{2k}$:
\begin{enumerate}
    \item Let $\text{dom}(e_1) = \bigcup_{j=1}^k \{2j-1, 2j\}$. We define the action to collapse each edge onto its odd endpoint: 
    \[ (2j-1)e_1 = 2j-1 \quad \text{and} \quad (2j)e_1 = 2j-1 \]
    \item Let $\text{dom}(e_2) = \bigcup_{j=1}^k \{2j, 2j+1\}$ (with addition modulo $n$). We define the action to collapse each edge onto its even endpoint:
    \[ (2j)e_2 = 2j \quad \text{and} \quad (2j+1)e_2 = 2j \]
\end{enumerate}
\end{definition}

Both $e_1$ and $e_2$ map every edge in their domain to a single vertex. Furthermore, the complete preimage of any target vertex is precisely the single edge collapsed onto it, while the preimages of all other cells are empty. Thus, both maps satisfy Margolis-Rhodes continuity and are valid idempotents in $E(St(C_{2k}))$.

\begin{lemma}
The element $f = e_1 e_2$ generates a cyclic subgroup of order $k$ inside $\langle E(St(C_{2k})) \rangle$.
\end{lemma}
\begin{proof}
Let $f = e_1 e_2$ acting on an arbitrary even vertex $2m$:
\[ (2m)e_1 = 2m-1 \]
We then apply $e_2$ to $2m-1$. Noting that $2m-1$ is part of the edge $\{2m-2, 2m-1\}$ in the domain of $e_2$, it collapses to its even endpoint:
\[ (2m-1)e_2 = 2m-2 \]
Thus, $(2m)f = 2m-2 \pmod{n}$. The function $f$ maps the set of even vertices bijectively onto itself via a cyclic shift of $-2$. Because there are $k$ even vertices, $f$ operates as a permutation cycle of length $k$. Hence, $f^{k+1} = f$ and $f^2 \neq f$ (since $k \ge 2$), confirming that $f$ generates a subgroup isomorphic to $\mathbb{Z}_k$.
\end{proof}

\subsection{The Odd Case ($n = 2k + 1$)}
Let $n = 2k + 1$ for $k \ge 2$. Because the cycle is odd, a perfect matching is impossible. Instead, an idempotent of maximum rank in $St(C_{2k+1})$ consists of exactly $k$ disjoint edges, leaving one vertex uncovered. We require three idempotents to successfully cycle the active components around the  ``slack'' of the odd vertex.

\begin{definition}[Idempotents for Odd Cycles]\label{oddcycle}
We define three partial idempotents $e_1, e_2, e_3 \in St(C_{2k+1})$:
\begin{enumerate}
    \item Let $\text{dom}(e_1) = \bigcup_{j=1}^k \{2j-1, 2j\}$. (Vertex $n$ is uncovered). Collapse to the right:
    \[ (2j-1)e_1 = 2j \quad \text{and} \quad (2j)e_1 = 2j \]
    \item Let $\text{dom}(e_2) = \bigcup_{j=1}^k \{2j, 2j+1\}$. (Vertex $1$ is uncovered). Collapse to the right:
    \[ (2j)e_2 = 2j+1 \quad \text{and} \quad (2j+1)e_2 = 2j+1 \]
    \item Let $\text{dom}(e_3) = \left( \bigcup_{j=1}^{k-1} \{2j+1, 2j+2\} \right) \cup \{n, 1\}$. (Vertex $2$ is uncovered). We collapse the first $k-1$ edges to the left, and the boundary edge to the right:
    \[ (2j+1)e_3 = 2j+1, \quad (2j+2)e_3 = 2j+1 \]
    \[ (n)e_3 = 1, \quad (1)e_3 = 1 \]
\end{enumerate}
\end{definition}

As in the even case, all vertex preimages are exactly single valid edges, confirming $e_1, e_2, e_3 \in E(St(C_{2k+1}))$.

\begin{lemma}
The element $f = e_1 e_2 e_3$ generates a cyclic subgroup of order $k$ inside $\langle E(St(C_{2k+1})) \rangle$.
\end{lemma}
\begin{proof}
We track the composition $f$ on the  set of odd vertices $\{1, 3, \dots, 2k-1\}$. We analyze two cases:

\textbf{Case 1: Intermediate vertices} (Let $x = 2j-1$ for $1 \le j \le k-1$). 
\[ (2j-1)e_1 = 2j \]
\[ (2j)e_2 = 2j+1 \]
\[ (2j+1)e_3 = 2j+1 \]
Thus, $(2j-1)f = 2j+1$.

\textbf{Case 2: The boundary vertex} (Let $x = 2k-1$).
\[ (2k-1)e_1 = 2k \]
\[ (2k)e_2 = 2k+1 = n \]
\[ (n)e_3 = 1 \]
Thus, $(2k-1)f = 1$.

The element $f$ acts strictly as the cyclic permutation $(1 \mapsto 3 \mapsto 5 \mapsto \dots \mapsto 2k-1 \mapsto 1)$ on its domain. Since the cycle length is $k = (n-1)/2 \ge 2$, $f^{k+1} = f$ but $f^2 \neq f$. Thus, it generates a subgroup isomorphic to $\mathbb{Z}_k$.
\end{proof}

\begin{theorem}
For any $n \ge 4$, the idempotent-generated subsemigroup $\langle E(St(C_n)) \rangle$ of the strict Margolis-Rhodes monoid is not aperiodic.
\end{theorem}
\begin{proof}
By Lemmas 2 and 4, the subsemigroup generates a non-trivial subgroup $\mathbb{Z}_{\lfloor n/2 \rfloor}$. Therefore the idempotent generated subsemigroup is  non-aperiodic.
\end{proof}

\begin{theorem}\label{cyclecomp}
The Margolis-Rhodes monoid $MR(C_n)$ of a cycle graph $C_n$ satisfies:
\[ MR(C_n)c = 2 \iff n \ge 4 \]
\end{theorem}
\begin{proof}

It is a fundamental property of the Krohn-Rhodes complexity operator that if a submonoid or a quotient (and hence a divisor) $S$ of a monoid $M$ satisfies $Sc = d$, then $Mc \ge d$. 

Because the principal factor $J^0$ is a divisor of an ideal of $MR(C_n)$, the small monoid $\operatorname{Sm}(J) = U(M) \cup J^0$ naturally divides the full monoid $MR(C_n)$. For $n \ge 4$, we established $\operatorname{Sm}(J)c = 2$, which implies $MR(C_n)c \ge 2$. By Theorem \ref{thesis}, $MR(C_n)c \le 2$ and we conclude that $MR(C_n)c = 2$. 

To complete the classification, we must verify that the complexity strictly drops for $n < 4$. 

For $n = 3$, the cycle graph $C_3$ is the complete graph $K_3$. On $K_3$, the only continuous partial functions that are not partial injective maps are those of rank 1, which strictly send a single edge to a vertex. Therefore, the monoid $MR(C_3)$ can be decomposed as the union of the symmetric inverse monoid on 3 elements, $\operatorname{SIS}(3)$, and an ideal consisting of elements of rank at most 1. Because all inverse semigroups have Krohn-Rhodes complexity at most 1, and the 0-minimal ideal of rank $\le 1$ maps is purely aperiodic, it follows that $\operatorname{MR}(C_{3})c = 1$. Since $C_{1}$ and $C_2$ are induced subgraphs of $C_3$, the results follow from Theorem \ref{sec:local_monoids}. 
This completes the proof.

\end{proof}

 Recall that the girth of a graph $\Gamma$ is the size of the largest induced cycle in $\Gamma$. Since the continuous functions on an induced subgraph of $\Gamma$ is a subsemigroup of $\operatorname{MR}(\Gamma)$ by Theorem \ref{sec:local_monoids} the following corollary is immediate.

\begin{corollary}\label{girhtcomp}

Let $\Gamma$ be a graph of girth at least 4. Then $\operatorname{MR}(\Gamma)c=2$.
    
\end{corollary}

A graph whose girth is less than 4 is known as a chordal graph. That is every cycle of length at least 4 has a chord, that is, at least one more edge between vertices of the graph. Chordal grpahs are very important in graph theory and its applications. Theorem \ref{cyclecomp} and Corollary \ref{girhtcomp} might lead one to conjecture that the complexity of a chordal graph has complexity at most 1. This is not true. We prove in the following sections that a path of length at least 14 has Margolis-Rhodes monoid of complexity 2. The proof is complicated. We begin this analysis in the next section.

\section{The Idempotent Generated Subsemigroup of \texorpdfstring{$MR(P_n)$}{MR(Pn)} is Aperiodic}
\label{sec:idempotents_aperiodicity}

We have completely determined when the complexity  of a cycle $C_n$ has complexity 2 in Theorem \ref{cyclecomp}. The case of complexity of paths is more difficult. We now begin its analysis. In this section, we show that the idempotent generated subsemigroup of $MR(P_{n})$ is aperiodic. Consequently, for every regular $\mathcal{J}$-class of $MR(P_{n})$, the small monoid $Sm(J)$ has complexity at most 1. Our first goal is to establish that every idempotent in $MR(P_n)$ preserves this natural linear order. We begin with a lemma concerning the localized proximity of a vertex to its image under an idempotent.

\begin{lemma}
\label{lem:proximity}
If $e \in MR(P_n)$ is an idempotent, then for every vertex $z \in \operatorname{Dom}(e)$, the vertex $z$ and its image $ze$ are either identical or adjacent in $P_n$. That is, $|z - ze| \le 1$.
\end{lemma}

\begin{proof}
Let $z \in \text{Dom}(e)$ and let $w = ze$. Because $w \in \text{Im}(e)$ and $e$ is an idempotent, $e$ acts as the identity on its image, meaning $we = w$. This implies that both $z$ and $w$ belong to the full preimage of the singleton simplex $\{w\}$ under $e$:
\[
z \in(\{w\})e^{-1} \quad \text{and} \quad w \in (\{w\})e^{-1}.
\]
By the definition of continuity for the Margolis-Rhodes monoid, the inverse image $(\{w\})e^{-1}$ of a vertex simplex must be a valid simplex in  $P_n$ or empty. Since it contains both $z$ and $w$, these two vertices must belong to the same simplex. Because the only simplices in a path graph are single vertices or single edges, $z$ and $w$ cannot be separated by a distance greater than $1$. Thus, $|z - ze| \le 1$.
\end{proof}

Using this locality property, we can now prove that all idempotents are order-preserving functions.

\begin{theorem}
\label{thm:idempotents_order_preserving}
Every idempotent $e \in E(MR(P_n))$ is an order-preserving partial function.
\end{theorem}

\begin{proof}
Suppose for the sake of contradiction that there exists an idempotent $e \in E(MR(P_n))$ that is not order-preserving. Then there must exist a pair of vertices $x, y \in \text{Dom}(e)$ such that $x < y$ but $xe > ye$. 

Let $u = xe$ and $v = ye$. Since $u$ and $v$ are integers and $u > v$, we have the inequality $u \ge v + 1$. Applying Lemma~\ref{lem:proximity} to both $x$ and $y$ yields:
\begin{enumerate}
    \item $|x - u| \le 1 \implies u - 1 \le x \le u + 1$
    \item $|y - v| \le 1 \implies v - 1 \le y \le v + 1$
\end{enumerate}
We string these inequalities together using our assumption that $x < y$, which implies $x \le y - 1$:
\[
u - 1 \le x < y \le v + 1
\]
Focusing on the outer bounds of this chain, we obtain $u - 1 < v + 1$, which simplifies to $u < v + 2$. Since $u$ and $v$ are integers, this strictly implies $u \le v + 1$. Combining this with our initial fact that $u \ge v + 1$, we are forced to conclude that $u = v + 1$.

Substituting $u = v + 1$ back into the inequality chain yields:
\[
v \le x < y \le v + 1
\]
The only integers $x$ and $y$ that can satisfy this strict configuration are $x = v$ and $y = v + 1$. Let us evaluate how the idempotent $e$ acts on these specific values:
\begin{itemize}
    \item Since $x = v$, its image is $ve = xe = u = v + 1$. 
    \item This means $v + 1 \in \text{Im}(e)$. Because $e$ is an idempotent, it must fix all elements in its image, giving $(v + 1)e = v + 1$.
    \item However, since $y = v + 1$, we also know by definition that $(v + 1)e = ye = v$.
\end{itemize}
Equating the two expressions for $(v + 1)e$ yields $v = v + 1$, a clear contradiction. Thus, no such pair $x < y$ with $xe > ye$ can exist, and $e$ is order-preserving.
\end{proof}

We now turn our attention to the submonoid generated by these idempotents, denoted $\langle E(MR(P_n)) \rangle$. It is well known that the idempotent generated subsemigroup of the monoid $PO_n$ of partial order preserving functions on an $n$-chain is aperiodic \cite{fernandes2005congruences}. The next theorem follows immediately from Theorem \ref{thm:idempotents_order_preserving}.

\begin{theorem}
The idempotent-generated submonoid $\langle E(MR(P_n)) \rangle$ of the Margolis-Rhodes monoid $MR(P_n)$ is aperiodic.
\end{theorem}

For the path graph $P_n$, the group of units $U(M)$ is the automorphism group of the graph. Thus the group of units is trivial for $n=1$ or isomorphic to the cyclic group of order $2$ for $n \ge 2$. 

\begin{theorem}
\label{thm:small_monoids_P_n}
For any path graph $P_n$ and any regular $\mathcal{J}$-class $J$ of $\operatorname{MR}(P_n)$, the corresponding small monoid satisfies $\operatorname{Sm}(J)c \le 1$.
\end{theorem}

\begin{proof}
Let $\operatorname{Sm}(J) = U(M) \cup J^0$ be the small monoid of $\operatorname{P_{n}}$ corresponding to the regular $\mathcal{J}$-class $J$. We have proved  that the idempotent generated subsemigroup of $\operatorname{Sm}(J)$ is aperiodic in . It follows from Theorem \ref{thm:tilson_small_monoids} that $\operatorname{Sm}(J)c \le 1$.

\end{proof}

\section{The Type II Subsemigroup of \texorpdfstring{$MR(P_n)$}{MR(Pn)}}
\label{sec:type_ii_analysis}

\subsection{Aperiodicity of \texorpdfstring{$MR(P_4)_{II}$}{MR(P4)II}}

In this section, we analyze the Type II subsemigroup $MR(P_n)_{II}$ of the Margolis-Rhodes monoid on path graphs. Despite the idempotent generated subsemigroup of $\MR(P_{n})$ is aperiodic we prove that its type II is aperiodic only for $n < 5$.

Recall that the type II submonoid $M_{II}$ of a monoid $M$ is the smallest monoid containing the idempotents and closed under weak conjugation: if $sts=s, s,t \in M$, then $sM_{II}t \cup tM_{II}s \subset M_{II}$. We prove that the Type II subsemigroup of $MR(P_{n})$ is aperiodic if and only if $n <5$.

Computations reveal that $MR(P_4)$ contains exactly 35 idempotents. Closing the idempotents under multiplication generates a subsemigroup $\langle E(MR(P_4)) \rangle$ of exactly 80 elements. To construct the full Type II subsemigroup $MR(P_4)_{II}$, we must close this idempotent-generated core under weak conjugation. 

Executing closure under weak conjugation over all available  pairs $(s, t)$ with $sts=s$ in $MR(P_4)$ captures exactly 9 additional elements that cannot be generated by pure idempotent product. This brings the complete size of the Type II subsemigroup $MR(P_4)_{II}$ to exactly 89 elements. 

By evaluating the power sequences of all 89 elements under multiplication, we observe that every sequence stabilizes within at most 4 iterations (meaning $x^m = x^{m+1}$ for $m \le 4$). Because there are no periodic orbits or non-trivial cyclic permutations, every maximal subgroup is trivial. Therefore, $MR(P_4)_{II}$ is aperiodic. For $n <4$, $\operatorname{MR}(P_{n})_{II}$ is a submonoid of $\operatorname{MR}(P_{4})_{II}$ by Theorem \ref{sec:local_monoids}. The authors maintain a copy of the calulations for $\operatorname{MR}(P_{4})_{II}$ for interested readers.

\subsection{Non-Aperiodicity of \texorpdfstring{$\operatorname{MR}(P_n)_{II},  n\geq 5$}{MR(P5)II}}
We now demonstrate that increasing the graph's length by a single vertex fundamentally alters the algebraic capacity of the monoid. Let $P_5$ be the path graph with 5 vertices, $V = \{1, 2, 3, 4, 5\}$, and 4 edges. 

\begin{theorem}\label{typeIIpath}
The Type II subsemigroup $MR(P_5)_{II}$ is not aperiodic.
\end{theorem}

\begin{proof}
Consider the partial map $x \in MR(P_5)$ defined by its action on the vertices:
\[
x = (- , 5, 5, - , 3)
\]
Evaluating the powers of $x$ yields:
\begin{align*}
    x^1 &= (- , 5, 5, - , 3) \\
    x^2 &= (- , 3, 3, - , 5) \\
    x^3 &= (- , 5, 5, -, 3) = x
\end{align*}
Since $x^3 = x$ but $x^2 \neq x$, the element $x$ generates a cyclic subgroup of order 2 (isomorphic to $\mathbb{Z}_2$) acting as a continuous orientation-reversing swap on the separated vertices $\{3, 5\}$. To prove $MR(P_5)_{II}$ is not aperiodic, we show $x \in MR(P_5)_{II}$.

We synthesize $x$ using Ash's generative components:
\begin{enumerate}
    \item  Let $k = (2, 2, 4, 4,  -)$. It is easy to see that $k^2 = k , k \in MR(P_5)_{II}$.
    \item \textbf{Regular Inverse Operators:} Let $s = (1, 5, 5, 3, 3)$. Analysis of its inverse images confirms $s$ is continuous. Since $s^3 = s$, the element acts as its own generalized regular inverse (i.e., setting $t = s$ satisfies $sts = s$).
    \item \textbf{Weak Conjugation:} Wrapping $k$ within the inverse operator pair $(s, s)$ forces the conjugated product into the Type II closure:
    \[
    f = s \cdot k \cdot s = (5, - , - , 3, 3) \in MR(P_5)_{II}
    \]
\end{enumerate}

 We now utilize the idempotent $k_2 = (2, 2, - , 4, 5)$ and the regular inverse $s_2 = (3, 1, 1, - , 5)$. Because $s_2 \cdot s_2 \cdot s_2 = s_2$, we apply weak conjugation to draw the shifted map $g$ into the Type II closure:
\[
g = s_2 \cdot k_2 \cdot s_2 = (- , 1, 1, - , 5) \in MR(P_5)_{II}
\]

Multiplying these two Type II elements yields:
\[
g \cdot f = (- , 5, 5, - , 3) = x
\]
Thus, $x \in MR(P_5)_{II}$. Because $x$ generates a non-trivial subgroup, $MR(P_5)_{II}$ is not aperiodic.
\end{proof}

We have not been able to determine the lower bound $\operatorname{MR}(P_{n})l$ from Theorem \ref{lb2}. In the subsequent sections, we use the theory of Aperiodic Flows \cite{Trans, flows, FlowsI} that will allow us to prove that $\operatorname{MR}(P_n)c=2$ for $n> 13$.

\section{Flow Theory}

A group-mapping  ($\operatorname{GM}$) semigroup $S$ is a semigroup that has a unique 0-minimal regular ideal $I(S) =M^{0}(G,A,B,C)$ such that $S$ acts faithfully on the left and the right of $I(S)$. Identifying a fixed $\mathcal{R}$-class with the set $G\times B$ we obtain a transformation semigroup $(G\times B,S)$. The action of $S$ induces an action on $B$. The faithful image of this action is the transformation semigroup $(B,\operatorname{RLM}(S))$ where $\operatorname{RLM}(S)$ is called the right letter mapping image of $S$. The importance of these concepts is the following theorem.

\begin{theorem}

Let $S$ be a semigroup. Then we have the following.

\begin{enumerate}

\item{$S$ has a $\operatorname{GM}$ image $T$ with $Sc=Tc$.}

\item{$Sc=\operatorname{RLM}(S)c$ or $Sc=1+\operatorname{RLM}(S)c$. Decidability of complexity can be reduced to answering which two of these alternatives hold.}

\end{enumerate}

\end{theorem}
 
Flows \cite{flows, Trans, FlowsI} provide a necessary and sufficient condition to compute the complexity of a $\operatorname{GM}$ semigroup by resolving the difference of complexity between the semigroup $S$ and its Right Letter Mapping image $\operatorname{RLM}(S)$. The \cite{complexity1, complexityn} results show that a $\operatorname{GM}$ semigroup $S$ has complexity at most $n$ if and only if its right letter mapping image $\operatorname{RLM}(S)$ has complexity $n$ and $S$ has a flow over a semigroup of complexity at most $n-1$.

The background in flow theory is extensive. In order to keep the ``flow" of the paper going, we've gathered the necessary background in appendices to this paper. The reader can consult \cite{Trans, flows, FlowsI} for precise details and proofs.

If we restrict the image of $f$ strictly to $G$-invariant cross-sections (i.e., $qf \in \operatorname{CS}(G \times B)$), the function $f$ constitutes a flow to the Rhodes Lattice $\operatorname{Rh}_B(G)$ \cite{FlowsI}.

In the context of the Margolis-Rhodes monoid, we are particularly interested in bounding the complexity below $2$. By \cite{Trans}, a $\operatorname{GM}$ semigroup $S$ has complexity $1$ if and only if its Right Letter Mapping image has complexity at most $1$, and there exists a complete flow from an entirely aperiodic transformation semigroup $T$ (where $Tc = 0$) to the Rhodes lattice associated with $S$. 

In the subsequent sections, we will deploy this theory by analyzing the aperiodic flows of specific induced subgraphs. By proving that no such aperiodic flow can exist for $3P_4$, and embedding $3P_4$ as an induced subgraph of $P_{14}$, we will use flow theory to prove that $MR(P_{14})c = 2$.

\section{An Important Example}

We define an example that allows us to show that $\operatorname{MR}(P_{n})c =2$ for $n>13$. Let $\mathbb{Z}_3 = \{1, g, g^2\}$ be the cyclic group of order 3. We define a 0-simple regular ideal $I(S) = \mathcal{M}^0(\mathbb{Z}_3, A, B, C)$ where the row index set is $A = \{a_1, \dots, a_7\}$ and the column index set is $B = \{1, 2, 3, 4\}$. The transpose of the structure matrix, $C^T: A \times B \to \mathbb{Z}_3^0$, is explicitly defined as:
\[
C^T = 
\begin{pmatrix} 
1 & 1 & 0 & 0 \\ 
0 & 1 & 1 & 0 \\ 
0 & 0 & 1 & 1 \\ 
1 & 0 & 0 & 0 \\ 
0 & 1 & 0 & 0 \\ 
0 & 0 & 1 & 0 \\ 
0 & 0 & 0 & 1 
\end{pmatrix}
\]
We construct the GM semigroup $S$ by adjoining three specific $\mathbb{Z}_3$-labeled partial transformations acting on the right of the columns $B$:
\begin{align*}
    z &= \{1 \to 2, \;\; 2 \to 1\} \\
    w &= \{1 \to 3, \;\; 3 \to 1\} \\
    t &= \{1 \to 1, \;\; 2 \to g3\}
\end{align*}

Generating $S = \langle I(S), z, w, t \rangle$ yields exactly \textbf{93 elements}. Any rank-1 map generated by the outer transformations operates identically to the base elements dictated by rows $a_4, a_5,$ and $a_6$ of the structure matrix, and is structurally absorbed directly into the 0-minimal ideal. The poset structure of Green's $\mathcal{J}$-classes is given below:

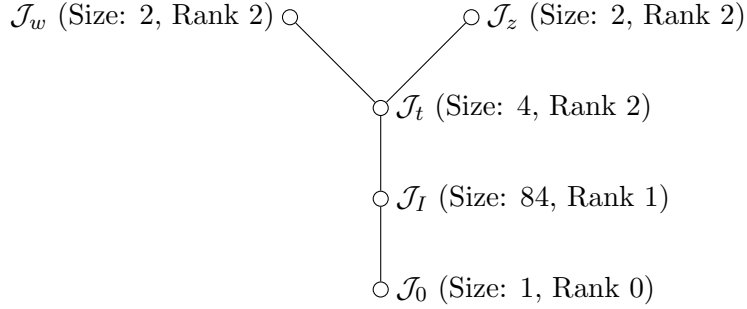
\begin{center}
\begin{tikzpicture}[scale=1.2, every node/.style={circle, draw, fill=white, inner sep=2pt, font=\small}]
    \node (J4) at (-1, 2) [label=left:$\mathcal{J}_w\ (\text{Size: 2, Rank 2})$] {};
    \node (J5) at (1, 2)  [label=right:$\mathcal{J}_z\ (\text{Size: 2, Rank 2})$] {};
    \node (J2) at (0, 1)  [label=right:$\mathcal{J}_t\ (\text{Size: 4, Rank 2})$] {};
    \node (JI) at (0, 0)  [label=right:$\mathcal{J}_I\ (\text{Size: 84, Rank 1})$] {};
    \node (J0) at (0, -1) [label=right:$\mathcal{J}_0\ (\text{Size: 1, Rank 0})$] {};
    \draw (J4) -- (J2); \draw (J5) -- (J2); \draw (J2) -- (JI); \draw (JI) -- (J0);
\end{tikzpicture}
\end{center}

\subsection{The Rhodes Lattice Flow Contradiction}

To determine if $c(S) = 1$, we evaluate whether an aperiodic flow can successfully filter the group tracking. Computations take place within the Evaluation Transformation Semigroup $\mathcal{E}(L)$, where $L = Rh_B(\mathbb{Z}_3)$. States are tracked as Subsets-Partitions-Cross sections (SPCs).

\begin{theorem}
The Evaluation Transformation Semigroup $\mathcal{E}(L)$ of $S$ encounters an unavoidable state collapse to the contradiction element $\Rightarrow\Leftarrow$, proving $S$ admits no aperiodic flow.
\end{theorem}
\begin{proof}
Initiate tracking from the stable singleton coordinate point $1$: $S_0 = \{1\} / \langle 1 \rangle$.
\begin{enumerate}
    \item \textbf{Loop Closure via $z$:} Applying $(z)^{\omega+*}$ merges coordinates 1 and 2 into a block with identical weights: $S_1 = \{1, 2\} / \langle 1, 1 \rangle$.
    \item \textbf{Projection via $t$:} Passing $S_1$ through the bridge $t = \{1 \to 1, 2 \to g3\}$ maps the active components to $\{1, 3\}$, injecting a phase distortion: $S_2 = \{1, 3\} / \langle 1, g \rangle$.
    \item \textbf{Orbit Shift via $w$:} Applying a single iteration of $w = (1 \;\; 3)$ swaps the internal weights: $S_2 \cdot w = \{1, 3\} / \langle g, 1 \rangle$.
    \item \textbf{Cross-Section Collision:} The loop closure $(w)^{\omega+*}$ requires evaluating the lattice join: 
    \[ S_3 = S_2 \vee (S_2 \cdot w) = (\{1, 3\} / \langle 1, g \rangle) \vee (\{1, 3\} / \langle g, 1 \rangle) \]
    A valid cross-section join over identical partition blocks exists if and only if the underlying cross-sections are proportional via some scalar $c \in \mathbb{Z}_3$. We must have $c \cdot 1 = g$ (thus $c = g$) and $c \cdot g = 1$. Substituting yields $g^2 = 1$. However, in $\mathbb{Z}_3$, $g^2 \neq 1$. Thus, no scaling factor exists, and the evaluation collapses to $S_3 = \;\; \Rightarrow\Leftarrow$.
\end{enumerate}
Because $S$ admits no aperiodic flow, its complexity must strictly exceed the RLM baseline: $c(S) = 1 + c(RLMR(S)) = 2$.
\end{proof}

\section{The Fiber Graph Structure}

Let the state space be $Q = \mathbb{Z}_3 \times \{1, 2, 3, 4\}$, yielding $|Q| = 12$. The fiber graph $\mathcal{F}$ of the transformation semigroup $(Q, S)$ has vertices $Q$ and edges defined by fibers of size 2.

\begin{lemma}
The fiber graph $\mathcal{F}$ of $(Q, S)$ is isomorphic to $3P_4$, the disjoint union of three paths of length 3. Furthermore, $(Q, S)$ is a $2$-transformation semigroup.
\end{lemma}
\begin{proof}
Elements of $I(S)$ map states to a single column $b$. Two states $(h_1, x_1)$ and $(h_2, x_2)$ collide under a matrix row $a$ if their accumulated phases match:
$ h_1 + C(x_1, a) = h_2 + C(x_2, a) \pmod 3 $.

Examining $C^T$, the non-zero rows $a_1, a_2, a_3$ have 1s (weight 0 in additive $\mathbb{Z}_3$) on adjacent columns. This creates edges exactly when $h_1 = h_2$ between columns 1-2, 2-3, and 3-4. No element collapses 3 states together, confirming $(Q, S)$ is a degree-2 transformation semigroup. Because edges only form between states of the same group phase, $\mathcal{F}$ decouples entirely into three disjoint components. Each component spans 4 vertices (length 3). Thus, $\mathcal{F} \cong 3P_4$.
\end{proof}

\section{Embedding into \texorpdfstring{$P_{14}$}{P14} and Complexity}

We adopt the standard convention that $P_n$ denotes a path graph of length $n-1$, possessing $n$ vertices. Therefore, $P_{14}$ possesses exactly 14 vertices.

\begin{theorem}
The Margolis-Rhodes Monoid $MR(P_{14})$ has a Krohn-Rhodes complexity of 2.
\end{theorem}
\begin{proof}
Because $S$ acts continuously on its fiber graph, $S \le MR(3P_4)$. We embed $3P_4$ into $P_{14}$ as an induced subgraph. $3P_4$ requires exactly 12 vertices. To keep the three path components disjoint within a single path, we insert exactly two "spacer" vertices, utilizing $12 + 2 = 14$ vertices total.

\begin{center}
\begin{tikzpicture}[scale=0.85, every node/.style={circle, draw, fill=white, inner sep=1.5pt, font=\footnotesize, minimum size=5mm}]
    \draw[thick] (1,0) -- (14,0);
    
    \node[fill=blue!20] (v1) at (1,0) {1}; \node[fill=blue!20] (v2) at (2,0) {2};
    \node[fill=blue!20] (v3) at (3,0) {3}; \node[fill=blue!20] (v4) at (4,0) {4};
    
    \node[draw=gray, dashed] (v5) at (5,0) {5};
    
    \node[fill=red!20] (v6) at (6,0) {6}; \node[fill=red!20] (v7) at (7,0) {7};
    \node[fill=red!20] (v8) at (8,0) {8}; \node[fill=red!20] (v9) at (9,0) {9};
    
    \node[draw=gray, dashed] (v10) at (10,0) {10};
    
    \node[fill=green!20] (v11) at (11,0) {11}; \node[fill=green!20] (v12) at (12,0) {12};
    \node[fill=green!20] (v13) at (13,0) {13}; \node[fill=green!20] (v14) at (14,0) {14};
\end{tikzpicture}
\end{center}

By mapping the components to $v_1 \dots v_4$, $v_6 \dots v_9$, and $v_{11} \dots v_{14}$ (skipping $v_5$ and $v_{10}$), $3P_4$ embeds perfectly into $P_{14}$. Since $3P_4$ is an induced subgraph of $P_{14}$, we have the division $S \le MR(3P_4) \prec MR(P_{14})$. By the monotonicity of complexity under division, $Sc \le MR(P_{14})c$. Since we proved $Sc = 2$, it follows that $MR(P_{14})c \ge 2$. 

Furthermore, because $(Q, S)$ is a $2$-transformation semigroup we have that $Sc \le 2$ by Theorem \ref{thesis}. By extension, because $P_{14}$ is a tree (acyclic), its graph monoid acts strictly as a $2$-transformation semigroup, enforcing $MR(P_{14})c \le 2$. Therefore, $MR(P_{14})c = 2$.
\end{proof}

\section{Complexity of Low-Order Path Graphs}
\label{sec:low_order_paths}

Having established that $MR(P_n)c = 2$ for $n \ge 14$ via the embedding of $3P_4$ and the theory of flows, we now resolve the complexity for path graphs of small order. Specifically, we demonstrate that the complexity strictly drops for paths of length $4$ or less.

\begin{theorem}
For $n \in \{1, 2, 3, 4\}$, the Margolis-Rhodes monoid of the path graph $P_n$ satisfies $MR(P_n)c \le 1$.
\end{theorem}

\begin{proof}
For $n=1$, the path graph $P_1$ consists of a single vertex. The monoid $MR(P_1)$ contains only the identity map and the empty map. As a semilattice, it is entirely aperiodic, yielding a trivial Krohn-Rhodes complexity of $MR(P_1)c = 0$.

For $n \in \{2, 3, 4\}$, we rely on our structural analysis of the Type II subsemigroup from Section~\ref{sec:type_ii_analysis}. We previously computed the complete Type II closure for $P_4$ and demonstrated that the subsemigroup $MR(P_4)_{II}$ is strictly aperiodic. Because the monoids for $P_2$ and $P_3$ embed naturally into $MR(P_4)$ as boundary-preserving retractions, their corresponding Type II subsemigroups are also completely aperiodic.

A Theorem of Rhodes and Tilson \cite{lowerbounds2} states that a finite semigroup $S$ divides a wreath product $A \wr G$, where $A$ is an aperiodic semigroup and $G$ is a finite group, if and only if its Type II subsemigroup $S_{II}$ is aperiodic.

Because $MR(P_n)_{II}$ is aperiodic for $n \le 4$, Ash's Theorem guarantees that $MR(P_n)$ divides a wreath product of the form $A \wr G$ for some aperiodic semigroup $A$ and group $G$. By the fundamental properties of the Krohn-Rhodes complexity operator, the complexity of a divisor is bounded by the complexity of the host, and the complexity of a wreath product is bounded by the sum of the complexities of its components:
\[ MR(P_n)c \le (A \wr G)c \le Ac + Gc \]
Since $A$ is aperiodic, $Ac = 0$. Since $G$ is a group, $Gc = 1$. Therefore, we conclude:
\[ MR(P_n)c \le 0 + 1 = 1 \]
This confirms that the complexity of $MR(P_n)$ is at most $1$ for $n \le 4$.
\end{proof}

This brings up the following problem.

\begin{problem}
Determine the exact Krohn-Rhodes complexity of $MR(P_n)$ for $5 \le n \le 13$.
\end{problem}

\begin{remark}

Using the algorithm of \cite{complexity1} the authors have a seven page document proving that $MR(P_{5})c=1$. This will appear in \cite{MasterList}. 

The algorithm in \cite{complexity1} must compute the evaluation semigroup (see the Appendices) of every regular $\mathcal{J}$-class of $\operatorname{MR}(P_{n})$. Since $P_{13}$ has 101 regular $\mathcal{J}$-classes, many of them very large, the computation is complicated. In order to press on with the paper, we leave this as a problem we will work on.
    
\end{remark}
\section{Translational Hulls and Margolis-Rhodes Monoids} \label{semiloc}

We recall a result from \cite{cremona} that the Margolis-Rhodes monoid $\operatorname{MR}(\Gamma)$ is the translational hull of an associated aperiodic $0$-simple semigroup. We let $SM(\Gamma) = St(\Gamma) \cup Aut(\Gamma)$ be the submonoid of $\operatorname{MR}(\Gamma)$ we obtain by adjoining the automorphsism group of $\Gamma$, which is the group of units of $\operatorname{MR}(\Gamma)$ to $St(\Gamma)$. We call this the {\em strict monoid} of $\Gamma$. We also recall that $SM(\Gamma)$ is the translational hull of an aperiodic 0-simple semigroup. For more details on translational hulls, see \cite[Section 5.5]{qtheory}

In this paper, we are only concerned with $0$-simple semigroups with trivial maximal semigroup. These are isomorphic to Rees matrix semigroups of the form $\mathcal{M}^{0}(1,A,B,C)$. It will be convenient for us to use ``inner product'' notation. We write $<b,a> = C(b,a)$. 
 
We define the translational hull to be the monoid of all pairs $(f,f^{*})$ where  $f \in PF_{R}(B)$ is a partial function acting on the right of $B$ and  $f^{*} \in PF_{L}(A)$ is a partial function acting on the left of $A$ that satisfy the ``linked equations"
\begin{equation}\label{trivlinked}<bf,a>=<b,f^{*}a>\end{equation}\label{linked} for all $b \in B, a \in A$ in inner product notation. Thus $f$ and $f^{*}$ are adjoint with respect to $<,>$.

We can view $C$ as the incidence matrix of an incidence system with points $B$ and blocks given by considering the row of $a \in A$ as
the subset $\overline{a}=\{b \in B|<b,a>=1\}$. Then membership in the translational hull, $<bf,a>=<b,f^{*}a>$ means that for all $b \in B, a \in A$, $bf \in \overline{a}$ 
if and only if $b \in \overline{f^{*}(a)}$. That is, $\overline{f^{*}(a)}$ is the inverse image $f^{-1}(\overline{a})$. Thus the blocks of the incidence system corresponding to $<,>$ are
 closed under inverse image with respect to $f$. This connection clearly motivates our monoids of continuous and strict continuous monoids on a graph. We make this
connection precise below. Viewing the translational hull of 0-simple semigroups over the trivial group as ``continuous'' partial maps on the corresponding incidence system has proved to be a fruitful connection between semigroup theory and combinatorics \cite{bibdtrans, bibdcont, projcont, AmigoWilson}.

Let $\Gamma = (V,E)$ be a graph. The graph incidence matrix of $V$ is the $|V| \times |E|$ matrix $S=S(\Gamma)$ whose entry in position $(v,e)$ is 1 if $v$ is a vertex of $e$ and 0 otherwise. Note that every column of $S$ has exactly two non-zero entries. The row sum of row $v$ is the degree of $v$ in $\Gamma$, where here, degree is used as in graph theory as the number of edges on which $v$ is a vertex. Thus there is a row of all zeros if and only if there is an isolated vertex in $\Gamma$. We will assume when talking about the graph incidence matrix that the graph has no isolated vertices. Since we are working with simple graphs (no loops nor multiple edges) all columns of $S$ are distinct. Two rows $v,w$ of $S$ are the same if and only if $vw$ is a connected component of $\Gamma$.

If we view $\Gamma$ as a simplicial complex, we have the simplicial incidence matrix, which is the $|V| \times |V \cup E|$ matrix $C = C(\Gamma)$ with entries 
$C(v,w)=1$ if and only if $v=w, v,w \in V$ and as above, $C(v,e)=1$ if and only if $v$ is a vertex of $e$.  As matrices, the relationship between $C$ and $S$ is that 
$C = [S|I_{V}]$, where $I_{V}$ is the $|V| \times |V|$ identity matrix. That is, we add $|V|$ columns to $S$ which contain the identity matrix in order to build $C$ from $S$. It is clear then that distinct columns and rows of $C$ are not equal to one another.

We use these matrices as the structure matrices of the following Rees matrix semigroups over the trivial group: $\mathcal{M}^{0}(1,E,V,S)$ and 
$\mathcal{M}^{0}(1,V \cup E,V,C)$. Assuming that $\Gamma$ has no isolated vertices means that these are regular Rees matrix semigroups and thus 0-simple semigroups. 
The discussion above leads immediately to the following result. See \cite{cremona} for a complete proof.

\begin{lemma} \label{transhull}

Let $\Gamma=(V,E)$ be a graph. Then the translational hull of $\mathcal{M}^{0}(1,V \cup E,V,C)$ is isomorphic to the Margolis-Rhodes monoid $\operatorname{MR}(\Gamma)$ of all continuous functions
on $\Gamma$. If $\Gamma$ has no isolated vertices, then the translational hull of $\mathcal{M}^{0}(1,E,V,S)$ is isomorphic to the monoid $SM(\Gamma)$ of strict continuous partial functions on $\Gamma$.
\end{lemma}

Lemma \ref{transhull} allows us to prove that two graphs are isomorphic if and only if their Margolis-Rhodes monoids are isomorphic.

\begin{theorem}

Let $\Gamma$ and $\Gamma'$ be graphs. Then $\Gamma$ is isomorphic to $\Gamma'$ if and only if $\operatorname{MR}(\Gamma)$ is isomorphic to $\operatorname{MR}(\Gamma')$.

\end{theorem}

\begin{proof}
Clearly if $\Gamma$ is isomorphic to $\Gamma'$, then $\operatorname{MR}(\Gamma)$ is isomorphic to $\operatorname{MR}(\Gamma')$. If $\operatorname{MR}(\Gamma)$ is isomorphic to $\operatorname{MR}(\Gamma')$, then by standard semigroup theory, the 0-minimal ideal $I(\operatorname{MR}(\Gamma))= \mathcal{M}^{0}(1,V \cup E,V,C)$ maps isomorphically onto the 0-minimal ideal $I(\operatorname{MR}(\Gamma'))= \mathcal{M}^{0}(1,V' \cup E',V',C')$. It follows that $|V|=|V'|$ and $|(V \cup E)= (V' \cup E')|$. For convenience we identify these sets. It follows \cite{CP} that there is a $|V \times V|$ permutation matrix $X$ and a $|(V \cup E) \times (V \cup E)|$ matrix $Y$ such that $C'=XCY$. By Lemma \ref{transhull} the incidence matrix of $\Gamma$ is that of $\Gamma'$ up to permutation of rows and columns. Therefore $\Gamma$ is isomorphic to $\Gamma'$.
\end{proof}

A similar argument proves the following.

\begin{theorem}

Let $\Gamma$ and $\Gamma'$ be graphs with no isolated vertices. Then $\Gamma$ is isomorphic to $\Gamma'$ if and only if $\operatorname{St}(\Gamma)$ is isomorphic to $\operatorname{St}\Gamma'$.
    
\end{theorem}
\section{Continuous Maps are Galois Connections on the Extended Face Poset}

In this section we formalize the adjointness displayed by the linked equations, Equation \ref{linked}, by showing that a partial vertex map on a graph is continuous in the sense of this paper if and only if its formally induced cell map on the extended face poset is order-preserving, atomic-to-top and satisfies the Principal Fiber Condition (PFC). Equivalently, this means that the induced map induces a Galois Connection, also called a residuated mapping, on the extended face poset. 

In particular, continuous maps on graphs in the sense of this paper are a certain class of continuous functions in the Alexandroff topology of the extended poset. Recall that the Alexandroff topology on a poset is the topology whose closed sets are the down ideals. A function on a poset is continuous in this topology if and only if it is order preserving. Thus the linked equations really describe an adjoint relationship in the appropriate seetting. We give the appropriate definitions and results.

\subsection{The Face Poset and Continuity}
Let $\Gamma = (V, E)$ be a simple graph. 
\begin{definition}
The \textbf{Face Poset} $P(\Gamma)$ is the set of all valid cells (vertices as singletons and edges as pairs) in $\Gamma$, ordered by set inclusion:
\[ X \le Y \iff X \subseteq Y \]
\end{definition}

We have defined a partial map on the vertices $f: V \to V$ to be \textbf{ continuous} if, for every valid target cell $C \in P(\Gamma)$, the vertex preimage:
\[ V_C = f^{-1}(C) = \{v \in V \mid vf \in C\} \]
is either a valid cell in $P(\Gamma)$ or the empty set $\emptyset$. The purpose of this section is to prove that a partial function on a graph is continuous in our sense if and only if an associated function is continuous in the Alexandroff topology of an extended face poset of the graph. Recall that the Alexandroff topology of a poset has downsets as closed sets.

\subsection{The Extended Poset and Induced Cell Map}
We extend the face poset.

\begin{definition}
The \textbf{The Extended Poset} is defined as $P(\Gamma)^\top = P(\Gamma) \cup \{\top\}$, where $\top$ will be used as a value that completes partial functions to total functions. We enforce $X \le \top$ for all $X \in P(\Gamma)^\top$.
\end{definition}

\begin{definition}
The \textbf{Induced Cell Map} $\hat{f}: P(\Gamma)^\top \to P(\Gamma)^\top$ maps a domain cell $C \in P(\Gamma)$ according to the following rules:
\begin{enumerate}
    \item If {\bf any} vertex $v \in C$ is not in the domain of $f$, then the cell is sent to the top: $C\hat{f} = \top$.
    \item If $C \subseteq \operatorname{Dom}(f) $ and $Cf \in P(\Gamma)$ (i.e., it forms a valid vertex or edge cell), then $C\hat{f} = Cf$.
    \item If $Cf \notin P(\Gamma)$, then $C\hat{f} = \top$.
    \item $\top\hat{f}=\top$
\end{enumerate}
\end{definition}

\subsection{Order Theoretic Axioms for Continuity}
We characterize continuity using two purely order-theoretic conditions on $\hat{f}$:
\begin{itemize}
    \item \textbf{Order-Preserving:} For all $X, Y \in P(\Gamma)$, if $X \le Y$, then $X\hat{f} \le Y\hat{f}$.
    \item \textbf{Principal Fiber Condition (PFC):} For every $C \in P(\Gamma)^{\top}$, the preimage of its principal down-set $\downarrow C\hat{f}^{-1} = \{X \in P(\Gamma)^{\top} \mid X\hat{f} \le C\}$ must exactly equal the principal down-set of some valid domain cell $\downarrow D$ (for $D \in P(\Gamma)^{\top}$), or be empty $\emptyset$.
\end{itemize}


\begin{theorem}\label{ordpres}
Let $\Gamma = (V, E)$ be a graph. A partial map on the vertices $f: V \to V$ is continuous if and only if its induced cell map $\hat{f}$ is order-preserving and satisfies the Principal Fiber Condition.
\end{theorem}

\subsection{Part I: Forward Proof \texorpdfstring{$(\implies)$}{(implies)}}
\begin{proof}
Assume $f$ is continuous. 

\textbf{Step 1: Prove $\hat{f}$ is Order-Preserving.}\\
Let $X \le Y$ in $P(\Gamma)^{\top}$. If $Y=\top$ then $Y\hat{f}=\top$ so we can assume that $Y \in P(\Gamma)$. This implies $X \subseteq Y$. We show $X\hat{f} \le Y\hat{f}$ via case analysis on $Y\hat{f}$:
\begin{itemize}
    \item \textbf{Case A:} $Y\hat{f} = \top$. Since $\top$ is the maximum element of $P(\Gamma)^\top$, the relation $X\hat{f} \le \top$ holds trivially.
    
    \item \textbf{Case B:} $Y\hat{f} = C \in P(\Gamma)$. By definition of 
    $\hat{f}$, $Y\subseteq \operatorname{Dom}(f)$
   and $Yf = C$ is a valid cell. Since $X \subseteq Y$, it follows that $X\subseteq \operatorname{Dom}(f)$  image satisfies $Xf \subseteq Yf = C$. Because $C$ is a valid cell (either a vertex or an edge), any non-empty subset of $C$ is structurally a valid cell in the graph. Thus, $Xf$ constitutes a valid cell $E \in P(\Gamma)$. Therefore, $X\hat{f} = E \le C = Y\hat{f}$.
    
\end{itemize}
Hence, order is strictly preserved.

\textbf{Step 2: Prove $\hat{f}$ satisfies the PFC.}\\
Let $C \in P(\Gamma)$ be a valid cell. Because $f$ is continuous, the vertex preimage $V_C = f^{-1}(C)$ is either empty or a valid cell $D \in P(\Gamma)$.
\begin{itemize}
    \item \textbf{Case A:} $V_C = \emptyset$. No vertex in the domain maps into $C$. Thus, $\downarrow C\hat{f}^{-1} = \emptyset$, matching the condition.
    \item \textbf{Case B:} $V_C = D \in P(\Gamma)$. We claim that$\downarrow C\hat{f}^{-1} = \downarrow D$:
    \begin{itemize}
        \item $(\downarrow D \subseteq \downarrow C\hat{f}^{-1})$: Let $X \in \downarrow D$, meaning $X \subseteq D$. For every $v \in X$, we have $v \in V_C$, so $vf \in C$. Thus, $Xf \subseteq C$. Because $C$ is a cell, $Xf$ forms a valid cell, yielding $X\hat{f} = Xf \le C$. Hence, $X \in \downarrow C\hat{f}^{-1}$.
        \item $(\downarrow C\hat{f}^{-1} \subseteq \downarrow D)$: Let $X \in \downarrow C\hat{f}^{-1}$, implying $X\hat{f} \le C$. Since $C \in P(\Gamma)$, $X\hat{f} \neq \top$. This implies that $X\hat{f} = Xf \subseteq C$. Therefore, for every $v \in X$, $vf \in C \implies v \in V_C = D$. Consequently, $X \subseteq D$, placing $X \in \downarrow D$.
    \end{itemize}
\end{itemize}
The Principal Fiber Condition is satisfied.
\end{proof}

\subsection{Part II: Backward Proof \texorpdfstring{$(\impliedby)$}{(impliedby)}}
\begin{proof}
Assume $\hat{f}$ is order-preserving and satisfies the Principal Fiber Condition.

We must show $f$ is continuous. Select any valid target cell $C \in P(\Gamma)$ and evaluate its vertex preimage $V_C = \{v \in V \mid vf \in C\}$. By the PFC axiom, the poset preimage $\downarrow C\hat{f}^{-1}$ evaluates to either $\emptyset$ or $\downarrow D$ for some valid domain cell $D \in P(\Gamma)$.

\begin{itemize}
    \item \textbf{Case A: $\downarrow C\hat{f}^{-1} = \emptyset$.} \\
    Suppose, for contradiction, that $V_C \neq \emptyset$. Then there exists a vertex $v \in V_C$ such that $vf \in C$. Consider the vertex cell $X = \{v\}$. Since $\{v\} \in P(\Gamma)$ and $f(\{v\}) = \{vf\} \subseteq C$, the induced map yields $\{v\}\hat{f}) \le C$. This places $\{v\} \in \downarrow C\hat{f}^{-1}$, contradicting that the set is empty. Thus, $V_C = \emptyset$.

    \item \textbf{Case B: $\downarrow C\hat{f}^{-1} = \downarrow D$ for some $D \in P(\Gamma)$.} \\
    We establish $V_C = D$. \\
    First, let $v \in V_C \implies vf \in C$. The vertex cell $X = \{v\}$ maps to $\{v\}\hat{f} = \{vf\} \le C$. Because it maps into $\downarrow C$, we have $\{v\} \in \downarrow C\hat{f}^{-1} = \downarrow D$. Since $\{v\} \in \downarrow D$, it follows that $\{v\} \subseteq D$, which implies $v \in D$. \\
    Second, let $v \in D$. Since $D \in P(\Gamma)$, $\{v\} \subseteq D \implies \{v\} \in \downarrow D$. By PFC, $\{v\} \in \downarrow C\hat{f}^{-1}$, so $\{v\}\hat{f} \le C$. Since $C \neq \top$, it must be that $\{v\}\hat{f} \neq \top$ and $vf$ is defined it follows that $vf \in C$. Thus, $v \in V_C$.
    
    Because $V_C = D \in P(\Gamma)$, the vertex preimage of $C$ is a valid cell.
\end{itemize}
Every valid target cell is a cell or empty vertex preimage; thus, $f$ is continuous.
\end{proof}

\subsection{ Characterization of Realizable Continuous Maps}

We now characterize which poset maps on $P(\Gamma)^{\top}$ are of the form $\hat{f}$ for some continuous map $f:V \rightarrow V$.

\begin{theorem}[Classification Theorem]\label{classificiation}
An order preserving poset map $g: P(\Gamma)^{\top} \to P(\Gamma)^\top$ is the induced cell map of a  continuous partial vertex function $f:V \rightarrow V$ if and only if it satisfies the following two conditions:
\begin{enumerate}
    \item \textbf{Atom or Top Preservation:} For every vertex cell $\{v\} \in P(\Gamma)$, its image $(\{v\})g$ is either a singleton vertex cell $\{w\} \in P(\Gamma)$ or the terminal object $\top$.
    \item \textbf{The Principal Fiber Condition (PFC):} For every  $C \in P(\Gamma)^{T}$, the poset preimage $(\downarrow C)g^{-1}$ is either empty  or a principal down-set $\downarrow D$ for some valid cell $D \in P(\Gamma)$.
\end{enumerate}
\end{theorem}

\begin{proof}
$(\implies)$ Let $g = \hat{f}$ for a continuous partial vertex map $f$. Condition 2 holds directly by Theorem 1. Condition 1 holds by the dimensionality constraints established in Section 5.2.

$(\impliedby)$ Let $g$ satisfy Conditions 1 and 2. 

\textbf{Step 1: Reconstruction of the Vertex Base.}\\
Because $g$ satisfies Atom Preservation, we can unambiguously construct a well-defined vertex partial function $f: V \to V$ operating exclusively on the atoms of the poset:
\[ f(v) = \begin{cases} w & \text{if } (\{v\})g = \{w\} \\ \text{Undefined} & \text{if } (\{v\})g = \top \end{cases} \]

\textbf{Step 2: Order-Preservation.}\\
We show that Condition 2 (PFC) naturally guarantees that $g$ is order-preserving without requiring it as an independent axiom. Let $X \le Y$ in $P(\Gamma)$, meaning $X \subseteq Y$. 
\begin{itemize}
    \item If $g(Y) = \top$, the relation $g(X) \le \top$ is satisfied trivially.
    \item If $g(Y) = C \in P(\Gamma)$, then by definition $Y \in (\downarrow C)g^{-1}$. By Condition 2, $(\downarrow C)g^{-1} = \downarrow D$ for some cell $D$. Because $Y \in \downarrow D \implies Y \subseteq D$. Since $X \subseteq Y$, it follows by transitivity that $X \subseteq D \implies X \in \downarrow D$. Therefore, $X \in (\downarrow C)g^{-1} \implies g(X) \le C = g(Y)$. 
\end{itemize}
Thus, order is preserved globally.

\textbf{Step 3: Verification of Higher-Dimensional Join Constraints.}\\
Since $g$ is order-preserving and defines a partial function at the vertex level, the behavior of any higher-dimensional edge cell $\{v_1, v_2\}$ is tightly bounded by its constituent atoms:
\begin{itemize}
    \item If any atom maps to $\top$ (e.g., $g(\{v_1\}) = \top$), then since $\{v_1\} \le \{v_1, v_2\}$, order-preservation forces $\top \le g(\{v_1, v_2\}) \implies g(\{v_1, v_2\}) = \top$. Destruction propagates upward automatically.
    \item If the atoms map to valid coordinates $g(\{v_1\}) = \{a\}$ and $g(\{v_2\}) = \{b\}$, the image $g(\{v_1, v_2\})$ must be an upper bound to both $\{a\}$ and $\{b\}$. If $\{a,b\}$ is a valid edge in $\Gamma$, the principal fiber constraint forces $g(\{v_1, v_2\}) = \{a,b\}$. If $\{a,b\}$ is disconnected in the target graph , no such valid bounded cell exists, forcing $g(\{v_1, v_2\}) = \top$.
\end{itemize}
Consequently, $g$ behaves exactly as the formally induced map $\hat{f}$. By Theorem 1, since $\hat{f}$ satisfies the PFC, the underlying reconstructed vertex map $f$ is guaranteed to be topologically continuous.
\end{proof}

\section{Pseudovarieties, Quasivarieties, and Open Problems}
\label{sec:pseudovarieties_and_open_problems}

The combinatorial and topological origins of the Margolis-Rhodes monoids $\operatorname{MR}(\Gamma)$ make them a rich source of structural examples in the broader theory of finite semigroups. We conclude this paper by examining the pseudovariety $\mathbf{V}(\mathcal{MR})$ and the quasivariety $\mathbf{Q}(\mathcal{MR})$ generated by these monoids, establishing a fundamental closure property regarding group extensions, and posing several open problems.

\subsection{ Transformation Semigroups and Wreath Products with Groups }


 We now demonstrate that this degree bound is preserved under left wreath products with permutation groups.

\begin{theorem}
Let $(Q, S)$ be a transformation semigroup of degree at most $2$, and let $(P, G)$ be a permutation group. Then the wreath product $(P, G) \wr (Q, S)$ also has degree at most $2$. 
\end{theorem}

\begin{proof}
The wreath product $(P, G) \wr (Q, S)$ acts on the Cartesian product space $P \times Q$. An element of this wreath product is represented as a pair $(f, s)$, where $s \in S$ and $f : Q \rightarrow G$ is a mapping. The right action on a state $(p, q) \in P \times Q$ is given by the standard coordinate-wise rule:
\[ (p, q)(f, s) = (p \cdot qf, q \cdot s) \]

To determine the degree of $(P, G) \wr (Q, S)$, we must bound the size of the full preimage of an arbitrary target point $(p', q') \in P \times Q$ under the action of $(f, s)$. Suppose $(p, q)(f, s) = (p', q')$. This yields the following system of equations:
\begin{enumerate}
    \item $q \cdot s = q'$
    \item $p \cdot qf = p'$
\end{enumerate}

From the first equation, $q$ must be an element of the inverse image $(q')s^{-1}$. Because the underlying transformation semigroup $(Q, S)$ has degree at most $2$, there are at most $2$ valid elements $q \in Q$ satisfying this condition.

For each such fixed $q$, the element $qf$ is a specific group element in $G$. Because $G$ acts as a permutation group on $P$, the action of $qf$ is a bijection on $P$. Therefore, the equation $p \cdot qf = p'$ possesses a unique solution for $p$, which is explicitly given by $p = p' \cdot qf^{-1}$.

Because there are at most $2$ valid choices for $q$, and each valid $q$ uniquely determines a single valid $p$, there are at most $2$ pairs $(p, q)$ in the complete preimage of $(p', q')$. Thus, the degree of the wreath product $(P, G) \wr (Q, S)$ is bounded by $2$.
\end{proof}

As  Margolis-Rhodes monoids are of degree $2$, and this property is preserved under left group extensions, we obtain an immediate structural corollary regarding their generated varieties.

\begin{corollary}
The pseudovariety $\mathbf{V}(\mathcal{MR})$ and the quasivariety $\mathbf{Q}(\mathcal{MR})$ generated by the Margolis-Rhodes monoids are closed under semidirect products with the pseudovariety of finite groups $\mathbf{G}$ on the left. That is, $\mathbf{G} * \mathbf{V}(\mathcal{MR}) = \mathbf{V}(\mathcal{MR})$ and $\mathbf{G} * \mathbf{Q}(\mathcal{MR}) = \mathbf{Q}(\mathcal{MR})$
\end{corollary}

\subsection{Open Problems}

While we have bounded the complexity of $\mathbf{V}(\mathcal{MR})$ by $2$ showing that it is contained in $\mathbf{C}_2$, the pseudovariety of semigroups of complexity $\le 2$, we do not have an exact characterization of this pseudoveriety. We pose the following open problems to drive future research:

\begin{problem}[The Finite Basis Problem]
Is the pseudovariety $\mathbf{V}(\mathcal{MR})$ finitely based? That is, can it be defined by a finite set of pseudoidentities, or does the underlying combinatorial geometry of the graphs necessitate an infinite sequence of structural invariants?
\end{problem}

\begin{problem}[Decidability of Membership]
Given an arbitrary finite semigroup $S$, is it decidable whether $S \in \mathbf{V}(\mathcal{MR})$? Furthermore, is membership decidable for the  quasivariety $S \in \mathbf{Q}(\mathcal{MR})$?
\end{problem}

\begin{problem}[The Quasivariety-Pseudovariety Gap]
Is $\mathbf{Q}(\mathcal{MR})$ properly contained in  $\mathbf{V}(\mathcal{MR})$? 
\end{problem}

\section{Generalization to Degree \texorpdfstring{$k$}{k}}
\label{sec:generalization_degree_k}

The concepts explored in this paper for degree 2 transformation semigroups and graph-based Margolis-Rhodes monoids naturally extend to higher dimensions. We can generalize the notion of continuous functions on graphs to continuous functions on arbitrary simplicial complexes.

\begin{definition}
A transformation semigroup $(Q, S)$ has \emph{degree at most $k$} if every fiber of any element of $S$ has size at most $k$. That is, for all $q \in Q$ and $s \in S$, the inverse image satisfies $|q s^{-1}| \le k$.
\end{definition}

To contextualize this geometrically, let $\Sigma$ be a simplicial complex of dimension $k$ (meaning that the largest simplex in $\Sigma$ contains $k+1$ vertices) with vertex set $V$. A partial function $f: V \rightarrow V$ is defined to be \emph{continuous} on $\Sigma$ if the inverse image of each simplex in $\Sigma$ under $f$ is either empty or forms a valid simplex in $\Sigma$.

The structural relationship between algebraic degree and geometric continuity on simplices was foundational to resolving early complexity bounds. We recall the following definitive theorems established by Margolis \cite{Tilsonnumber}.

\begin{theorem}[\cite{Tilsonnumber}]
A finite semigroup $S$ has a faithful representation by partial functions of degree at most $k$ if and only if it embeds into the monoid of continuous partial functions on a simplicial complex of dimension $k$.
\end{theorem}

This embedding theorem bounds the complexity.

\begin{theorem}[\cite{Tilsonnumber}]
If the degree of a transformation semigroup $(Q, S)$ is $k$, then the Krohn-Rhodes complexity of $S$ satisfies $Sc \le k$.
\end{theorem}

In this paper, our intensive study of $\operatorname{MR}(\Gamma)$ strictly operates in the $k=2$ paradigm, as graphs are $1$-dimensional simplicial complexes (where simplices are at most edges, having 2 vertices). However, the methodologies developed here—specifically the construction of strict subsemigroups, the localized complexity analysis of right-orbit small monoids via Tilson's Theorem, and the bridging of $\mathcal{J}$-classes through the theory of aperiodic flows—generalize to the case of arbitrary degree $k$. 

By replacing the graph $\Gamma$ with a $k$-dimensional simplicial complex, these techniques provide a unified framework for bounding and exactly computing the Krohn-Rhodes complexity of higher-dimensional topological transformation semigroups.

\subsection{Conclusion}
In this paper, we have explored the structural and dynamical properties of the Margolis-Rhodes monoid $\operatorname{MR}(\Gamma)$ of continuous partial functions on finite graphs. We showed that the poset of regular $\mathcal{J}$-classes of $\operatorname{MR}(\Gamma)$ is isomorphic to the poset of induced subgraphs of $\Gamma$. We established that $\operatorname{MR}(\Gamma)$ has complexity at most $2$ while its injective submonoid $Inj(\Gamma)$ and its subsemigroup $St(\Gamma)$ of strict continuous functions have complexity at most $1$. 

We applied Tilson's Theorem to the right-orbit monoids of regular $\mathcal{J}$-classes, proving that the monoid of a cycle graph $MR(C_n)$ has complexity $2$ if and only if $n \ge 4$ . Finally, by using the theory of Aperiodic Flows, we overcame the limitations of Tilson's Theorem for path graphs, showing $MR(P_n)c = 2$ for $n \ge 14$ and framing the algebraic-geometric threshold that restricts low-order paths to a complexity of $1$.

We showed that $\operatorname{MR}(\Gamma)$ is isomorphic to the translational hull of the aperiodic 0-simple semigroup whose structure matrix is the incidence matrix of the simplicial complex associated to $\Gamma$. We also showed that $\operatorname{MR}(\Gamma)$ is isomorphic to a certain monoid of Galois connections on the extended face poset of $\Gamma$. We used this to prove that Margolis-Rhodes monoids are isomorphism invariants for graphs.

Together, these results bridge combinatorial graph theory, topology, and the structure theory of finite semigroups.

\section{AI Acknowledgment}

By AI in this section we mean Gemini Pro. This was provided to the first author by the Department of Mathematics, Bar-Ilan University. We began using AI in May, 2026, although we had a version of the paper earlier in 2026 before the use of AI. As AI is a new tool for most mathematicians in 2026, and there is no universally accepted way to acknowledge it, we are being honest and ethical about its use.

AI was used as the equivalent of a graduate assistant and colleague. Much of what it did was to save us a lot of time by writing LaTex code and by doing computations in Python. If the authors had access to a colleague with more expertise than ourselves in LaTeX and Python, we would have asked them for help. 

We discussed with AI some of the theorems and computations in the paper. We list here the most important results we got from AI.

\begin{enumerate}

\item{After teaching AI the definition of continuous functions on a graph from our papers \cite{cremona, deg2part2}, it called the monoid the ``Margolis-Rhodes monoid on a graph", which we accepted.}

\item{We discussed Theorem \ref{strictreg} with AI. If $f$ is a regular element of $\StG{\Gamma}$, then an $\mathcal{R}$ equivalent idempotent gives the Transition Condition and an $\mathcal{L}$ equivalent idempotent gives the Covering Condition. Thus these two conditions are necessary for regularity of $\StG{\Gamma}$. We asked if regularity was equivalent to the existence of one of these conditions. It produced a counterexample and gave a proof of what became Theorem \ref{strictreg}.}

\item{AI produced the counterexample in Proposition \ref{not reg}.}

\item{AI gave a proof of what became Theorem \ref{bipartcover}.}

\item{In Definition \ref{oddcycle}, we had a collection of $n$ idempotents for the odd cycle $C_n$ that generated a non-trivial group element. AI improved this to 3 idempotents.}

\item{We had the proof of Theorem \ref{thm:idempotents_order_preserving} as it appears. Nonetheless we asked AI if the idempotent generated subsemigroup of $MR(P_{n})$ is aperiodic. It immediately came up with the proof that idempotents in this monoid are partially order preserving and that the monoid of partial order preserving functions is known to be aperiodic. This result was impressive as we didn't guide it to the order preserving property nor suggested that the partial order preserving monoid is aperiodic.}

\item{AI computed the element in Theorem \ref{typeIIpath} that is not aperiodic, and proved that 5 was optimal. This result would have been difficult to discover by hand.}

\item{After teaching AI the definition of the monoid $\operatorname{MR}(\Gamma)$, it suggested the results of Theorem \ref{ordpres} and Theorem \ref{classificiation}. Its original proofs were incorrect, but we worked together to get a correct proof. The authors are not order theorists and we were surprised by the results- we had not ask AI for such a characterization. Nor did we find such a characterization in the literature. We like this result as it shows that continuity in our sense on a graph (or more generally a simplicial complex) is equivalent to continuity in the sense of topology of an associated function in the Alexandroff topology of the extended face poset.}
\end{enumerate}
\appendix

\section{APPENDIX: Flows and the Flow Decomposition Theorem}\label{Flows}

We review the definition of flows from  an automaton with alphabet $X$  to the set-partition $SP(G\times B)$ and to the Rhodes lattice $Rh_{B}(G)$.  We summarize material from \cite{Trans, AHNR.1995, flows, FlowsI, complexity1, qtheory}. For more details, see these references. 

\subsection{The Set-Partition Lattice and the Rhodes Lattice of a Group-Mapping Semigroup }
\label{PLFlows}

The Presentation Lemma and Flow Theory give a necessary and sufficient condition for a $\operatorname{GM}$ semigroup $S$ to have the property that $Sc = \operatorname{RLM}(S)c$. It was shown to be decidable in \cite{complexity1, complexityn}.
 
Let $S$ be a $\operatorname{GM}$ semigroup with $0$-minimal ideal $M^{0}(G,A,B,C)$. The definition of flow gives a map from the state set 
of an automaton to a lattice associated with $S$. In the literature, there are two such classes of lattices. The first is the set-partition lattice $\operatorname{SP}(G \times B)$. This is the lattice 
whose elements are all pairs $(Y,\Pi)$, where $Y$ is a subset of $G\times B$ and $\Pi$ is a partition on $Y$. Here $(Y,\Pi) \leqslant (Z,\Theta)$ if 
$Y \subseteq Z$ and for all $y \in Y$, the $\Pi$ class of $y$ is contained in the $\Theta$ class of $y$. 

The second lattice is the Rhodes lattice $\operatorname{Rh}_{B}(G)$. We review the basics. For more details see~\cite{AmigoDowling}. 
Let $G$ be a finite group and $B$ a finite set. A partial partition on $B$ is a partition $\Pi$ on a subset $I$ of $B$. We also consider the 
collection of all functions $F(B,G)$, $f \colon I \rightarrow G$ from subsets $I$ of $B$ to $G$. The group $G$ acts on the left of $F(B,G)$ by $(gf)(b) = gf(b)$
for $f \in F(B,G), g \in G, b \in \operatorname{Dom}(f)$. An element $\{gf \mid f \colon I \rightarrow G, I \subseteq X, g \in G\}$ of the quotient set 
$F(B,G)/G$ is called a cross-section 
with domain $I$. It should be thought of as a projectification of a cross-section of the projection from $I$ to $B$ in the usual topological sense.

An $\operatorname{SPC}$ (Subset, Partition, Cross-section) over $G$ is a triple $(I,\Pi,f)$, where $I$ is a subset of $B$, $\Pi$ is a partition of $I$ and $f$ is a collection of 
cross-sections one for each $\Pi$-class $\pi$ with domain $\pi$. If the classes of $\Pi$ are $\{\pi_{1},\pi_{2}, \ldots, \pi_{k}\}$, then we sometimes write 
$\{(\pi_{1},f_{1}), \ldots, (\pi_{k},f_{k})\}$, where $f_{i}$ is the cross-section associated to $\pi_{i}$. For brevity we denote this set of cross-sections by 
$[f]_{\pi}$.  We let $\operatorname{Rh}_{B}(G)$ denote the set of all $\operatorname{SPC}$s on $B$ over the group $G$ union a new element 
$\Longrightarrow\Longleftarrow$  that we call {\em contradiction} and is the top element of the lattice structure on $\operatorname{Rh}_{B}(G)$.
Contradiction occurs because the join of two SPCs need not exist. In this case we say that the contradiction is their join.

The partial order on $\operatorname{Rh}_{B}(G)$ is defined as follows. We have $(I,\pi,[f]_{\pi}) \leqslant (J,\tau,[h]_{\tau})$ if:
\begin{enumerate}
\item
{$I \subseteq J$;} 
\item
{Every block of $\pi$ is contained in a (necessarily unique) block of $\tau$;} 
\item{if the $\pi$-class $\pi_{i}$ is a subset of the $\tau$-class $\tau_{j}$, then the restriction of $h$ to $\pi_{i}$ equals $f$ restricted to $\pi_{i}$ as elements of $F(B,G)/G$. That is, $[h|_{\pi_{i}}] = [f|_{\pi_{i}}]$.}
\end{enumerate}

 See \cite[Section 3]{AmigoDowling} for the definition of the lattice structure on $\operatorname{Rh}_{B}(G)$. The underlying set of the Rhodes lattice 
$\operatorname{Rh}_{B}(G)$ minus the contradiction is the set underlying the Dowling lattice on the same set and group. The Dowling lattice 
has a different partial order. For the connection between Rhodes lattices and Dowling lattices see \cite{AmigoDowling}.

We note that $\operatorname{SP}(G \times B)$ is isomorphic to the Rhodes lattice $\operatorname{Rh}_{G\times B}(1)$ of the trivial group over the 
set $G \times B$. We need only note that a cross-section to the trivial group is a partial constant function to the identity and can be omitted, leaving 
us with a set-partition pair. There are no contradictions for Rhodes lattices over the trivial group, and in this case the top element is the pair 
$(G \times B,(G \times B)^{2})$. Despite this, we prefer to use the notation $\operatorname{SP}(G \times B)$ instead of $\operatorname{Rh}_{G\times B}(1)$.

Conversely, we can find a copy of the meet-semilattice of $\operatorname{Rh}_{B}(G)$ as a meet subsemilattice of $\operatorname{SP}(G\times B)$. 
We begin with the following important definition.

\begin{definition}
A subset $X$ of $G \times B$ is a cross-section if whenever $(g,b),(h,b) \in X$, then $g = h$. That is, $X$ defines a cross-section of the projection 
$\theta \colon X \rightarrow B$. Equivalently, $X^{\rho} \subseteq B \times G$, the reverse of $X$, is the graph of a partial function 
$f_{X} \colon B \rightarrow G$.  An  element $(Y,\Pi) \in \operatorname{SP}(G \times B)$ is a cross-section if  every partition class $\pi$ of $\Pi$ 
is a cross-section.  
\end{definition}

From the semigroup point of view, a cross-section is a partial transversal of the $\mathcal{H}$-classes of the distinguished $\mathcal{R}$-class, 
$R = G \times B$ of a $\operatorname{GM}$ semigroup. That is, the $\mathcal{H}$-classes of $R$ are indexed by $B$ and a cross-section picks at most one 
element from each $\mathcal{H}$-class. 

An element $(Y,\Pi)$ that is not a cross-section is called a {\em contradiction}. That is, $(Y,\Pi)$ is a contradiction if some $\Pi$-class $\pi$ contains 
two elements $(g,b),(h,b)$ with $g \neq h$. We note that the set of cross-sections is a meet subsemilattice of $\operatorname{SP}(G \times B)$ and 
the set of contradictions is a join subsemilattice of $\operatorname{SP}(G \times B)$.

We will identify $B$ with the subset $\{(1,b)\mid b \in B\}$ of $G \times B$, which as above we think of as a system of representatives of the 
$\mathcal{H}$-classes of the distinguished $\mathcal{R}$-class. Note then that $G\times B$ is the free left $G$-act on 
$B$ under the action $g(h,b)=(gh,b)$. This action extends to subsets and partitions of $G\times B$ . Thus $\operatorname{SP}(G\times B)$ 
is a left $G$-act. An element $(Y,\Pi)$ is {\em invariant} if $G(Y,\Pi) = (Y,\Pi)$. It is easy to see that $(Y,\Pi)$ is invariant if and only if:

\begin{enumerate}
 \item{$Y=G \times B'$ for some subset $B'$ of $B$.}  

 \item{For each $\Pi$-class $\pi$, $G\pi \subseteq \Pi$ and is a partition of $G\times B''$ where $B''\subseteq B'$.}
\end{enumerate}

Thus $(Y,\Pi)$ is invariant if and only if $Y=G\times B'$ for some subset $B'$ of $B$ and there is a partition $B_{1},\ldots, B_{n}$ of $B'$ such that for all 
$\Pi$ classes $\pi$, $G\pi$ is a partition of $G\times B_{i}$ for a unique $1 \leqslant i \leqslant n$. 

Let $\operatorname{CS}(G\times B)$ be the set of invariant cross-sections $(Y,\Pi)$ in $\operatorname{SP}(G \times B)$. $\operatorname{CS}(G\times B)$ is a 
meet-subsemilattice of $\operatorname{SP}(G\times B)$. We add a top-element $\Rightarrow\Leftarrow$ to $\operatorname{CS}(G \times B)$. The resulting structure $\operatorname{CS}^{\Rightarrow\Leftarrow}(G \times B)$ is a lattice if we define join to be its join in $\operatorname{CS}(G \times B)$ if it exists and $\Rightarrow\Leftarrow$ otherwise.

The following Proposition is proved in \cite{FlowsI}. It gives the connection between set-partition lattices and Rhodes lattices.

\begin{proposition}\label{Updown}
The set-partition lattice $\operatorname{SP}(G\times B)$ is isomorphic to the Rhodes lattice $\operatorname{Rh}_{G\times B}(1)$. 
The Rhodes lattice $\operatorname{Rh}_{B}(G)$ is isomorphic to the lattice $\operatorname{CS}^{\Rightarrow\Leftarrow}(G\times B)$. 
\end{proposition}

\subsection{The Type II Subsemigroup and the Tilson Congruence}\label{redux.sec}

In this subsection we review the type II subsemigroup $S_{II}$ of a semigroup. It plays an important role in finite semigroup theory and a central role in flow theory. It was first defined in \cite{lowerbounds2}. In that paper it was proved that it is decidable if a regular element of a semigroup $S$ belongs to $S_{II}$. In particular, it followed that if $S$ is a regular semigroup, then membership in $S_{II}$ is decidable. Later 
Ash \cite{Ash} proved that membership in $S_{II}$ is decidable for all finite semigroups $S$.

The type II subsemigroup $S_{II}$ of $S$ is the smallest subsemigroup of $S$ containing all idempotents and closed under weak conjugation: if $xyx = x$ for $x,y \in S$, then $xS_{II}y \cup yS_{II}x \subseteq S_{II}$. Membership in $S_{II}$ is clearly decidable from this definition. Its importance stems from the following Theorem, where we give its original, not a priori decidable definition.

\begin{theorem}

Let $S$ be a finite semigroup. Then $S_{II}$ is the intersection of all subsemigroups of the form $1\phi^{-1}$ where $\phi:S \rightarrow G$
is a relational morphism from $S$ to a group $G$.

\end{theorem}

The aforementioned decidability results in \cite{lowerbounds2, Ash} prove that the two definitions we give define the same subsemigroup of a semigroup $S$. Here are some important properties of the type II subsemigroup. See \cite{qtheory} for proofs.

\begin{theorem}\label{typeIIprop}

\begin{enumerate}

\item{Let $f:S \rightarrow T$ be a morphism between semigroups $S$ and $T$. If $s \in S_{II}$, then $sf \in T_{II}$. Therefore the restriction of $f$ to $S_{II}$ defines a morphism $f_{II}:S_{II} \rightarrow T_{II}$. This defines a functor on the category of finite semigroups.}

\item{If $S$ divides $T$, then $S_{II}$ divides $T_{II}$.}

\item{Assume that a semigroup $S$ divides a semidirect product $T*G$ where $T$ is a semigroup and $G$ is a group. Then $S_{II}$ divides $T$.}

\end{enumerate}

\end{theorem}





A {\em congruence} on a transformation semigroup $(Q,S)$ is an equivalence relation $\approx$ on $Q$ such that if $q\approx q'$ and for $s \in S$, and both $qs$ and $q's$ are defined, then $qs\approx q's$. Every $s \in S$ defines a partial function on $\faktor{Q}{\approx}$ by 
$[q]_{\approx}s=[q's]_{\approx}$ if  $q's$ is defined for some $q' \in [q]_{\approx}$. The quotient $\faktor{(Q,S)}{\approx}$ has states $\faktor{Q}{\approx}$ and semigroup $T$ the semigroup generated by the action of all $s \in S$ on $\faktor{Q}{\approx}$. We remark that $T$ is not necessarily a quotient semigroup of $S$, but is in the case that $(Q,S)$ is a transformation semigroup of total functions. 

A congruence $\approx$ is called {\em injective} if every $s \in S$ defines a partial 1-1 function on $\faktor{Q}{\approx}$. It is easy to see that the intersection of injective congruences is injective. Therefore, there is a unique minimal injective congruence $\tau$ on any transformation semigroup $(Q,S)$. We call $\tau$ the Tilson congruence on a transformation semigroup because of the following proved in \cite{Redux}. It is central to the theory of flows.

\begin{theorem}\label{redux}

Let $(G \times B,S)$ be a $\operatorname{GM}$ transformation semigroup. Then the minimal injective congruence $\tau$ on $S$ is defined as follows: 
$(g,b) \tau (g',b')$ if and only if there are elements $s,t \in S_{II}$ such that $(g,b)s=(g',b')$ and $(g',b')t=(g,b)$.

\end{theorem}
 
\begin{remark}

\item{We can state this theorem by $(g,b)S_{II}=(g',b')S_{II}$. That is, $(g,b)$ and $(g',b')$ define the same ``right coset'' of $S_{II}$ on $G \times B$.}

\item{The proof in \cite{Redux} works on an arbitrary transitive transformation semigroup. We stated it in the case of $\operatorname{GM}$ transformation semigroups because that's how we will use it in this document.}

\item{Theorem \ref{redux} can be used to greatly simplify the proof in \cite{lowerbounds2} for decidability of membership in $S_{II}$ for regular elements of an arbitrary  semigroup.}

\end{remark}

For later use we record the following corollary to Theorem \ref{redux}.

\begin{corollary}

Let $(G \times B,S)$ be a $\operatorname{GM}$ transformation semigroup. Then $(g,b) \tau (g',b')$ if and only if there are elements 
$s,t \in S_{II} \cap I(S)$ such that $(g,b)s=(g',b')$ and $(g',b')t=(g,b)$. That is, we can choose the elements $s$ and $t$ in Theorem \ref{redux} to be in the 0-minimal ideal of $S$.

\end{corollary}

\begin{proof} The condition is sufficient by Theorem \ref{redux}. Conversely, assume that $(g,b)\tau (g',b')$. Then there are elements $s,t \in S_{II}$ such that $(g,b)s=(g',b')$ and $(g',b')t=(g,b)$. Since $I(S)$ is a 0-simple semigroup, there are idempotents $e,f \in I(S)$ such that 
$(g,b)e=(g,b)$ and $(g',b')f=(g',b')$. Therefore, $(g,b)es=(g,b)s=(g',b')$ and similarly $(g',b')ft =(g,b)$. Since $e,f$ are idempotents we have $es,ft \in S_{II}$. As $I(S)$ is a 0-minimal ideal, we also have $es,ft \in I(S)$. 

\end{proof}

%
%

\subsection{Definition of Flows and Their Properties}
Let $(G \times B,S)$ be the transformation semigroup associated to a $\operatorname{GM}$ semigroup $S$ and let $X$ be a generating set for $S$. By a deterministic automaton we mean an automaton such that each letter defines a partial function on the state set.

\begin{definition}

Let $\mathcal{A}$ be a deterministic automaton with state set $Q$ and alphabet $X$. 
A {\em flow} to the lattice $SP(G\times B)$ on $\mathcal{A}$ is a function $f:Q \rightarrow SP(G \times B)$ such that for each $q \in Q, x \in X$, with $qf =(Y,\pi)$ and $(qx)f = (Z,\theta)$ we have:

\begin{enumerate}

\item{For all $(g,b) \in G \times B$, there is a $q \in Q$ such that $(g,b) \in qf$.}
 
\item{$Yx \subseteq Z$.}

\item{Multiplication by $x$ considered as an element of $S$ induces a partial 1-1 map from $Y/\pi$ to $Z/\theta$.

That is, for all $y,y' \in Y$ we have $y,y'$ are in a $\pi$-class if and only if $yx,y'x$ are in a $\theta$-class whenever $yx,y'x$ are both defined.}

\item{{\bf The Cross Section Condition:} For all $q \in Q$, $qf$ is a cross-section. That is, for all $g,h \in G, b \in B$ if $(g,b),(h,b) \in qf$ it follows that $g=h$.}

\end{enumerate}
\end{definition}

We use Proposition \ref{Updown} to give a definition of a flow to the Rhodes lattice $Rh_{B}(G)$.

\begin{definition}

Let $\mathcal{A}$ be a deterministic automaton with state set $Q$ and alphabet $X$. 
A {\em flow} to the lattice $Rh_{B}(G)$ is a flow to $SP(G\times B)$ such that for each state $q$, $qf \in CS(G\times B)$. That is, $qf$ is a $G$-invariant cross-section.
   
\end{definition}

\begin{remark}

    It follows from the original statement of the Presentation Lemma \cite{AHNR.1995} that if $S$ has a flow with respect to some automaton over $SP(G\times B)$, then it has a flow from the same automaton over $Rh_{B}(G)$. The proof of the equivalence of the Presentation Lemma and Flows in Section 3 of \cite{Trans} works as well for the Presentation Lemma in the sense of \cite{AHNR.1995}. 

\end{remark}

We will use all the terminology and concepts from \cite{Trans}. If $\mathcal{A}$ is an automaton with alphabet $X$ and state set $Q$, recall that its completion is the automaton $\mathcal{A}^{\square}$ that adds a sink state $\square$ to $Q$ and declares that $qx = \square$ if $qx$ is not defined in $\mathcal{A}$. A flow on $\mathcal{A}$ is a complete flow if on $\mathcal{A}^{\square}$ extending $f$ by letting $\square f = (\emptyset, \emptyset)$ the bottom of both the set-partition and Rhodes lattices remains a flow and furthermore, for all $(g,b) \in G \times B$, there is a $q\in Q$ such that $(g,b) \in Y$, where $qf=(Y,\Pi)$. All flows in this paper will be complete flows.

 Flows are related to the Presentation Lemma \cite{AHNR.1995}, \cite[Section 4.14]{qtheory}. They give a necessary and sufficient condition for $Sc =\operatorname{RLM}(S)c$ where $S$ is a $\operatorname{GM}$ semigroup.  See Section 3 of \cite{Trans} for a proof of the following Theorem

\begin{theorem}\label{PLflow}[The Presentation Lemma-Flow Version]

Let $(G \times B,S)$ be a $GM$ transformation semigroup with $S$ generated by $X$. Let $k> 0$ and assume that $\operatorname{RLM}(S)c = k$.  Then $Sc = k$ if and only if there is an $X$ automaton $\mathcal{A}$ whose transformation semigroup $T$ has complexity strictly less than $k$ and a complete flow $f:Q \rightarrow Rh_{B}(G)$.


\end{theorem}

 Flows were preceded by the Presentation Lemma \cite{AHNR.1995}, \cite[Section 4.14]{qtheory}. The Presentation Lemma was shown to be equivalent to the existence of an appropriate flow in \cite[Section 3]{Trans}. There are three versions of the Presentation Lemma and its relation to Flow Theory in the literature \cite{AHNR.1995}, \cite[Section 4.14]{qtheory}, \cite{Trans}. These have very different formalizations and it is not clear how to pass from one version to another. The terminology is different as well. For example, the definition of cross-section in each of these references, as well as the one we use in this paper is different.  The following Theorem gives a unified approach to these topics. It summarizes known results in the literature and is meant to emphasize the strong connections between slices, flows and the presentation lemma and the corresponding direct product decomposition. \cite{FlowsI} for a proof. We use some of the results that we've proved in previous sections. For background on the derived transformation semigroup and the derived semigroup theorem see \cite{Eilenberg, qtheory}.

\begin{theorem} \label{uniform}

Let $(G \times B,S)$ be $\operatorname{GM}$ and assume that $\operatorname{RLM}(S)c \leq n$. Then the following are equivalent:

\begin{enumerate}

\item{$Sc \leq n$.}

\item{There is an aperiodic relational morphism $\Theta:S \rightarrow H\wr T$, where $H$ is a group and $Tc \leq n-1$.}

\item{There is a relational morphism $\Phi:S \rightarrow T$ where $Tc \leq n-1$ and such that the Derived Transformation semigroup $D(\Phi)$ is in $Ap*Gp$.}

\item{There is a relational morphism $\Phi:S \rightarrow T$ where $Tc \leq n-1$ such that the Tilson congruence $\tau$ on the Derived Transformation semigroup $D(\Phi)$ is a cross-section.}

\item{$(G \times B,S)$ admits a flow from a transformation semigroup $(Q,T)$ with $(Q,T)c \leq n-1$.}

\item{$S \prec (G \wr (Sym(B)) \wr T) \times \operatorname{RLM}(S)$ for some transformation semigroup $T$ with $Tc \leq n-1$.}

\end{enumerate}

\end{theorem}

%
%

We emphasize the case for complexity 1 and aperiodic flows as this is used in many of our examples.

\begin{theorem} \label{Apuniform}

Let $(G \times B,S)$ be $\operatorname{GM}$ and assume that $\operatorname{RLM}(S)c \leq 1$. Then the following are equivalent:

\begin{enumerate}

\item{$Sc=1$.}

\item{There is an aperiodic relational morphism $\Theta:S \rightarrow H\wr T$, where $H$ is a group and $T$ is aperiodic.}

\item{There is a relational morphism $\Phi:S \rightarrow T$ where $T$ is aperiodic such that the Derived transformation semigroup $D(\Phi)$ is in $Ap*Gp$.}

\item{There is a relational morphism $\Phi:S \rightarrow T$ where $T$ is aperiodic such that the Tilson congruence $\tau$ on the Derived transformation semigroup $D(\Phi)$ is a cross-section.}

\item{$(G \times B,S)$ admits an aperiodic flow.}

\item{$S \prec (G \wr (Sym(B)) \wr T) \times \operatorname{RLM}(S)$ for some aperiodic semigroup $T$.}

\end{enumerate}

\end{theorem}

\section{APPENDIX The Flow Monoid and the Evaluation Transformation Semigroup}\label{Eval}

\subsection{The Monoid of Closure Operations}\label{Closops}

In this Appendix we gather definitions and results from \cite{Trans}. For more details the reader is strongly urged to consult this paper. Since all the computations we do on examples in this document are done in the Evluation Transformation Semigroup defined in Section 5 of \cite{Trans} our goal is to summarize the background material in that paper needed to define this object. We begin with the definition of the monoid of closure operations on the direct product $L^{2} = L \times L$ of a lattice $L$ with itself. 

Let $L$ be a lattice and let $L^{2} = L \times L$. Let $f$ be a closure operator on $L^2$. By definition this means that $f$ is an order preserving, extensive (that is, for all $(l_{1},l_{2}) \in L^{2}, (l_{1},l_{2}) \leqslant (l_{1},l_{2})f$), idempotent function on $L^2$. A {\em stable pair} for $f$ is a closed element of $f$. Thus a stable pair $(l,l') \in L^2$ is an element such that $(l,l')f = (l,l')$.  The stable pairs of $f$ are a meet closed subset of $L \times L$. Conversely, each meet closed subset of $L \times L$ is the set of stable pairs for a unique closure operator on $L \times L$. We will identify $f$ as a binary relation on $L$ whose pairs are precisely the stable pairs. Let $B(L)$ be the monoid of binary relations on the set $L$.  As is well known, $B(L)$ is isomorphic to the monoid $ M_{n}(\mathcal{B})$ of $n \times n$ matrices over the 2-element Boolean algebra $\mathcal{B}$, where $n=|L|$. The Boolean matrix associated to $f$ is of dimension $|L| \times |L|$. It has a 1 in position $(l_{1},l_{2})$ if $(l_{1},l_{2})$ is a stable pair and a 0 otherwise. 

With this identification, the collection $\mathcal{C}(L^{2})$ of all closure operators on $L^2$ is a submonoid of the monoid $B(L)$ of binary relations on $L$~\cite[Proposition 2.5]{Trans}. We thus also consider $\mathcal{C}(L^{2})$ to be a monoid of $|L| \times |L|$ Boolean matrices. Let $L$ be either $Rh_{B}(G)$ or $\operatorname{SP}(G\times B)$. We now recall the definition of some important unary operations on $\mathcal{C}(L^{2})$. 
\begin{definition}

Let Let $f \in \mathcal{C}(L^{2})$.

\begin{enumerate}

\item{The domain of $f$ denoted by $\operatorname{Dom}(f)$ is the set $\{x \mid \exists y, (x,y) \in f\}.$}

\item{Define the relation $\overleftarrow{f}$ by $\overleftarrow{f}= \{(x,x) \mid x \in \operatorname{Dom}(f)\}$. $\overleftarrow{f}$ is called {\em back flow along $f$}. See \cite[Remark 2.26]{Trans} for the reason for this terminology.}

\item{The relation $f^{*}$ is defined by $f^{*} = f \cap \{(x,x) \mid x \in X\}$. $f^{*}$ is called {\em the Kleene closure of $R$}. See Section 2 and Section 4 of \cite{Trans} for the reason for this terminology.}

\item{Define the {\em loop of $f$} to be the relation $f^{\omega+*}=f^{\omega}f^{*}$, where $f^{\omega}$ is the unique idempotent in the subsemigroup generated by $f$.}

\end{enumerate}

\end{definition}

At times it is convenient to identify $\overleftarrow{f}$ as $1|_{\operatorname{Dom}(f)}:L \rightarrow L$, the identity function restricted to 
$\operatorname{Dom}(f)$. Similarly, we identify $f^{*}$ as $1|_{\operatorname{Fix{f}}}:L \rightarrow L$, the restriction of the identity to the set of fixed-points of $f$, where $x \in L$ is a fixed point if $(x,x) \in f$. Our use of these will be clear from the context.

\subsection{The 0-Flow Monoid}\label{FMon}

Let $S$ be a $\operatorname{GM}$ semigroup generated by $X$ and let $x \in X$. Let $I(S) = \mathcal{M}^{0}(A,G,B,C)$. We define a binary relation $f_{x}$
on $\operatorname{SP}(G \times B)$ by $((Y,\Pi), (Z,\Theta)) \in f_{x}$ if and only if $Yx \subseteq Z$ and the partial function induced by
right multiplication by $x$, $\cdot x: Y \rightarrow Z$ induces a well-defined partial injective map $\cdot x: Y/\Pi \rightarrow Z/\Theta$. This means that if  $(g,b), (g',b') \in Y$ and $(g,b)x,(g',b')x$ are both defined (and hence in $Z$ ), then $(g,b)\Pi (g',b')$ if and only if $(g,b)x\Theta (g',b')x$. Then $f_{x} \in \mathcal{C}(L^{2})$. See \cite[Proposition 2.22]{Trans}. $f_{x}$ is called the free-flow by $x$.

We now define the 0-flow monoid $M_{0}(L)$ as follows.

\begin{definition}\label{Flowops}

Let $L = \operatorname{SP}(G \times B)$. The 0-flow monoid $M_{0}(L)$, is the
smallest subset of $\mathcal{C}(L^{2})$ satisfying the following axioms:

\begin{enumerate}

\item {(Identity) The multiplicative identity $I$ of $\mathcal{C}(L^{2})$ is in $M_{0}(L)$.}

\item{(Points) For all $x \in X$, $f_{x}$ the free-flow along $x$ belongs to $M_{0}(L)$.}

\item{(Products) If $f_{1}, f_{2} \in M_{0}(L)$, then $f_{1}f_{2} \in M_{0}(L)$.}

\item{(Vacuum) If $f \in M_{0}(L)$, then $\overleftarrow{f} \in M_{0}(L)$.}\label{Vac}

\item{(Loops) If $f \in M_{0}(L)$, then $f^{\omega+*}\in M_{0}(L)$.}

\end{enumerate}

\end{definition}

We remark that for each $n \geq 0$, there is an $n$-flow monoid $M_{n}(L)$ defined in \cite{Trans}. The definition above of $M_{0}(L)$ is exactly what is defined in \cite{Trans} and used extensively in \cite{complexity1}. For $n>0$, Axioms (1)-(4) are the same for $M_{n}(L)$ as for $M_{0}(L)$. Axiom (5) restricts the use of the loop operator to $n$-loopable elements \cite[Section 4]{Trans}. In \cite{complexityn}, $n$-loopable elements are replaced by a more restrictive definition and the modified $M_{n}(L)$ plays a crucial role in the main results of \cite{complexityn}.

\subsection{The Evaluation Transformation Semigroup}\label{Ets}

We now defined the Evaluation Transformation Semigroup $\mathcal{E}(L) = (\operatorname{States}, \operatorname{Eval}(L))$ as in \cite{Trans}. Again, because of our interest in this paper on aperiodic flows, we don't define the $n$-Evaluation Transformation Semigroups $E_{n}$ for $n>0$. We begin with the definition of Well-Formed Formualae (WFFs).

\begin{definition}

Let $X$ be an alphabet. We define a well-formed formula inductively as follows.

\begin{enumerate}

\item{The empty string $\epsilon$ is a well-formed formula.}

\item{Each letter $x \in X$ is a well-formed formula.}

\item{If $\tau, \sigma$ are well-formed formulae, then
so is $\tau\sigma$.}

\item{If $\tau$ is a well-formed formula that is not a proper power (i.e., not of the form
$\sigma^{n}$ where $n > 1$), then $\tau^{\omega+*}$ is also a well-formed formula.} 

\end{enumerate}

\end{definition}

The set of well-formed
formulae is denoted by $\Omega(X)$. Well-formed formulae will be denoted by Greek
letters. As a convention, if $\tau=\sigma
^{n}$, where $\sigma$ is not a proper power, then we set
$\tau^{\omega+*}=\sigma^{\omega+*}$. In other words, we extract roots before applying the unary operation $\omega+*$.

Let $V = \prod_{f \in M_{0}(L)}\overleftarrow{f}$. $V$ is called the {\em Vacuum}. See \cite{Trans} for an explanation of this terminology. In \cite{Trans}, $V$ is denoted by $\mathcal{F}_{0}$. It is proved in \cite{Trans} that $V$ is an idempotent in $M_{0}(L)$. We now will work exclusively in the subsemigroup $VM_{0}(L)V$ of $M_{0}(L)$. We want to interpret WFFs in $VM_{0}(L)V$.

\begin{definition}

Define recursively a partial function $\mathcal{I}: \Omega(X)\rightarrow VM_{0}(L)V$ as follows. 

\begin{enumerate}

\item{$\epsilon\mathcal{I} = V$.}

\item{$x\mathcal{I} = VxV$ for $x \in X$.}

\item{If $\mathcal{I}$ is already defined on $\tau, \sigma \in \Omega(X)$, set $(\tau\sigma)\mathcal{I} = \tau\mathcal{I}\sigma\mathcal{I}$.}

\item{If $\tau \in \Omega(X)$ is not a proper power and $\tau\mathcal{I}$ is defined, set $\tau^{\omega+*}\mathcal{I} = 
(\tau\mathcal{I})^{\omega+*}$.}

\end{enumerate}

\end{definition}

We normally omit $\mathcal{I}$ and assume that a WFF $\tau$  is being evaluated in $VM_{0}(L)V$ according to the definition of $\mathcal{I}$. 
 We first define a new operator on elements of $M_{0}(L)$ called {\em forward flow}. Recall that the bottom of the lattice $\operatorname{SP}(G \times B)$ is the pair $\Box = (\emptyset, \emptyset)$.

\begin{definition}

Let $f \in M_{0}(L)$ and let $l \in L$. Let $(l,\Box)f =(l_{1},l_{2})$ We define the forward flow of $f$ denoted by $\overrightarrow{f}$ by 
$l\overrightarrow{f}=l_{2})$. That is, we apply $f$ to $(l,\Box)$ and project to the right-hand coordinate.

\end{definition}

The following is proved in Sections 2 and 4 of \cite{Trans}.

\begin{enumerate}

\item{$\overrightarrow{f}:L \rightarrow L$ is an order preserving function on $L$.}

\item{If $lV=l$, then $(l,\Box)f=(l,l')$ and $l' \in LV$. That is, the left-coordinate of $(l,\Box)f$ is still $l$ and the right-hand coordinate is in $LV$. Therefore $\overrightarrow{f}$ is a well-defined function from $LV$ to $LV$.}

\item{The assignment of $f$ to $\overrightarrow{f}$ defines an action of $VM_{0}(L)V$ on $LV$. It follows that we have a transformation semigroup $(LV,M'_{0}(L))$, where $M'_{0}(L)$ is the image of $VM_{0}(L)V$ on $LV$ under this action.}

\end{enumerate}

We now restrict the action in $(LV,M'_{0}(L))$ to the set of $\operatorname{States}$ defined as follows. Let $(g,b) \in G \times B$. Then 
the element $(\{(g,b)\},\{(g,b)\}) \in L$ is called a {\em point}. In Section 5 of \cite{Trans}, it is proved that every point $p$ satisfies $pV=p$. 

\begin{definition}

The set of $\operatorname{States}$ is the smallest subset of $LV$ such that:

\begin{enumerate}

\item{Every point $p \in \operatorname{States}$.}

\item{If $l \in \operatorname{States}$, then $l\overrightarrow{f} \in \operatorname{States}$.}

\end{enumerate}

\end{definition}

In other words $\operatorname{States}$ is the smallest subset of $LV$ containing the points and closed under the action of $M_{0}(L)$ on $LV$. We can finally define the Evaluation Transformation Semigroup where all the computations in the examples in this paper take place.

\begin{definition}

The Evluation Transformation Semigroup $\mathcal{E}(L)$ is defined by $\mathcal{E}(L) = (\operatorname{States}, \operatorname{Eval}(L))$ where $\operatorname{Eval}(L)$ is the image of $M'_{0}(L)$ by restricting its action to $\operatorname{States}$.

\end{definition}

We remark that there is an Evaluation Transformation Semigroup $\mathcal{E}_{n}(L)$ for all $n \geq 0$. The definition here is the case $n=0$
which is all we need in this paper as we are only concerned with aperiodic flows.

\subsection{Notation for Elements of the Rhodes Lattice $Rh_{B}(G)$}\label{SPCNotation}

In the examples  we compute using the Evaluation Transformation Semigroup $\mathcal{E}(L) = (\operatorname{States}, \operatorname{Eval}(L))$ where $L$ is the Rhodes lattice $Rh_{B}(G)$. We fix the notation we use for $SPC$ in our examples. Let $(W,\pi,[\mu]_\pi) \in Rh_B(G)$ be an element in the Rhodes lattice. Let $\pi_1,\pi_2,\ldots,\pi_k$ be the equivalence classes of
$\pi$ and let $\mu_i = \mu|_{\pi_i}: \pi_i \to G$ for $i\in \{1,2,\ldots,k\}$. Write $\pi_{1}\mid \pi_{2}\mid \ldots \pi_{k}/\langle \mu_1 \mid \mu_2 \mid \ldots \mid \mu_k\rangle$ for $(W,\pi,[\mu]_\pi)$. For readability, we drop set-brackets around elements of a partition class.
\begin{example}
\label{example.downstairs}
Let $G=\{1,-1\}=Z_2$, $W=\{1,2,3,4,5,6\}$, and
$\pi=\{1\mid  2 \ 5 \mid  3 \ 4 \ 6\}$. Let $\mu: \{1,2,3,4,5,6\} \to Z_2$ be defined by
$\mu(1)=\mu(3)=\mu(5)=1$ and $\mu(2)=\mu(4)=\mu(6)=-1$. Then we write $\pi=\{1 \mid  2 \ 5 \mid 3 \ 4 \ 6\}/\langle1 \mid \ -1 \ 1 \mid \ 1 \ -1 \ -1\rangle$
\end{example}

\bibliographystyle{plain}
\bibliography{stubib}

\end{document}